\documentclass[10pt]{article}
\usepackage{latexsym,amsfonts,amsthm,amsmath,epsfig,amssymb}
\usepackage{amsfonts}
\usepackage{amsthm}
\usepackage{amsmath,amssymb}
\usepackage{booktabs}
\usepackage{graphicx}

\newtheorem{thm}{Theorem}[section]
\newtheorem{cor}{Corollary}[section]
\newtheorem{lemma}{Lemma}[section]

\newtheorem{example}{Example}[section]
\newcounter{nextauthor}
\def\mathrm{\mbox}

\numberwithin{remark}{section}

\begin{document}
\title{\Large {\bf Globalized distributionally robust optimization problems under the moment-based framework}\thanks{This work was supported by the National Natural Science Foundation of China (11471230, 11671282) and the Fundamental Research Funds for the Central Universities, Southwest Minzu
University (2018NQN28) and Southwestern University of Finance and Economics (JBK 2002006, JBK1805001).}}
\author{
{ Ke-wei Ding}$^{a}$ { Nan-jing Huang}$^{b}$\footnote{To whom all correspondences should be addressed: nanjinghuang@hotmail.com; njhuang@scu.edu.cn}\, and { Lei Wang}$^{c}$\\
$^a$\, {\small School of Computer Science and Technology, Southwest Minzu University,}\\
{\small Chengdu, Sichuan 610041, P.R. China}
\\
$^b$\, {\small Department of Mathematics, Sichuan University,}\\
{\small Chengdu, Sichuan 610064, P.R. China}
\\
$^c$\, {\small School of Economic Mathematics,  Southwestern University of Finance and Economics,}\\
{\small Chengdu, Sichuan 611130, P.R. China}
}
\date{}
\maketitle
\vspace*{0mm}
\begin{center}
\begin{minipage}{5.9in}
{\bf Abstract.}
This paper is devoted to reduce the conservatism of distributionally robust optimization with moments information.
Since the optimal solution of distributionally robust optimization is required to be feasible for all uncertain distributions
in a given ambiguity distribution set and so the conservatism of the optimal solution is inevitable.
To address this issue, we introduce the globalized distributionally robust counterpart (GDRC) which allows constraint violations controlled by functional distance of the true distribution to the inner uncertainty distribution set. We obtain the deterministic equivalent forms for several GDRCs under the moment-based framework. To be specific, we show the deterministic equivalent systems of inequalities for the GDRCs under second order moment information with a separable convex distance function and a jointly convex distance function, respectively.  We also obtain the deterministic equivalent inequality for the GDRC under first order moment and support information. The computationally tractable examples are presented for these GDRCs. A numerical tests of a portfolio optimization problem is given to show the efficiency of our methods and the results demonstrate that the globalized distributionally robust solutions is non-conservative and flexible compared to the distributionally robust solutions.

\vskip 0.1cm


\vskip 0.1cm

\end{minipage}
\end{center}

%%%%%%%%%%%%%%%%%%%%%%%%%%%%%%%%%%%%%%%%%%%%%%%%%%%%%%%%%%%%%%%%%%%%%%%%%%%%
\section{Introduction}
It is well known that the optimal solution of mathematical programming heavily depends on the value of parameters.
However, uncertainty of parameters arises from estimation errors or implementation errors can not be avoided in many real-world problems, and the optimal solution under inappropriate estimations may deliver a poor performance for the true optimization problem. For overcoming this issue, robust optimization (RO) which requires that the optimal solution must be feasible for all realizations in a given uncertainty set has been a popular method when the parameters do not admit a stochastic nature, for more details, please see \cite{BGK,BGN,FW,Sim,SSS} and the references given there. When the uncertain parameters are viewed as random variables with unknown distribution, an important way to deal with the ambiguity of distribution information is the distributionally robust optimization (DRO) method.

DRO which is famous as minimax stochastic optimization has been a significant framework for modeling optimization problems with uncertainty parameters. Most DROs are developed with the purpose of achieving a computationally tractable forms. Choosing suitable ambiguity distribution sets to get tractability is an important issue in DROs. Wide range of the ambiguity distribution sets are introduced and the corresponding distributionally robust optimization problems have been rewritten as (or approximated by) computationally tractable deterministic forms, see \cite{BBB,CSZ,EKDR,GLTS,JWMYZ,ShapiroD,WKS,XG} for more details. Using moments information estimated from history samples to construct the ambiguity distribution sets is an important way in DROs. This method has been pioneered by Scarf \cite{S} who studied a single-product newsvendor problem under an ambiguity set with known mean and variance. Goh and Sim \cite{GohSim} considered distributionally robust linear programs under ambiguity sets with support, mean, covariance and directional deviations information. Delage and Ye \cite{DY} demonstrated that distributionally robust stochastic program can be solved efficiently under the ellipsoid-type uncertainty set on mean and covariance. Natarajan et al. \cite{NMU} obtained tractable conic representation for a maximin expected utility model with several box-type uncertainty moments sets. Liu et al. \cite{LLZ} approximated a distributionally robust optimization of an emergency medical service station location and sizing problem as a parametric second-order conic representable program under ellipsoid-type uncertainty moments information set. Gourtani et al. \cite{GNX} reformulated  distributionally robust facility location problem  as a semi-definite program under known second order moment information and also reformulated the problem as a semi-infinite program under known first order moment information.

However, an inevitable issue is the optimal solution of DRO is extremely conservative since DRO takes the worst case scenario in the ambiguity information set which contains all the ``physically possible" realizations. One method to settle down the conservatism is to reduce the size of the ambiguity information set. But the optimistic solution produced by this way may deliver a poor performance since the true distribution may get out of the chosen distribution information set in practice. How to reduce the conservatism of the optimal solution and give more flexibility to the ambiguity set are not only an important issue in DRO, but also in RO. In RO field, Ben-Tal et al. \cite{BBHV} introduced the globalized robust counterpart which allowed controlled constraint violations in a larger uncertainty set to give more flexibility to the uncertainty set.
Compared with the classical robust counterpart, the globalized robust counterpart gives
the decision makers more flexibility to release the feasibility requirement of the uncertainty set in
a control way. The globalized robust counterpart originally introduced in \cite{BBN} has been named the comprehensive
robust counterpart, and also has been discussed in \cite{GBHB} for reducing the conservatism of the robust counterpart by viewing the outer and inner uncertainty sets mentioned in \cite{BBHV} as "physically possible" set and "normal range" set of realizations, respectively.

Since the working mechanism of DRO and RO are extremely analogical, it is natural to introduce globalized ideology to DROs for reducing the conservatism of DROs or releasing the feasibility requirement of the uncertainty distribution set. The main purpose of this paper is to develop a new method called the globalized distributionally robust optimization to reduce the conservatism of the optimal solution.

The rest of the paper is organized as follows. In the next section we introduce the globalized distributionally robust counterpart (GDRC) and refine it into a moment-based framework. After that Section 3 obtains the deterministic equivalent system of inequalities for the GDRC under assumptions that the first and second order moments belong to corresponding convex and compact sets respectively with a distance function that is separable about mean and covariance. In Section 4 we present the deterministic equivalent system of inequalities for the GDRC with a special function which is jointly convex in mean and covariance. Section 5 is addressed to obtain the deterministic equivalent inequality for the GDRC under assumptions that the first order moment set and the support set are convex and compact when the constraint function is concave in uncertain parameter. In Section 6, we present a numerical experiment to show the effectiveness of our method, before we summarize the results in Section 7.

{\bf Notation}

Let $f : \mathbb{R}^{k\times k}\rightarrow \mathbb{R}$. The convex conjugate of $f$ is defined as
$$f^*(Y)=\sup_{X\in {\rm dom}(f)}\{\langle X,Y\rangle -f(X) \},$$
where $\langle X,Y\rangle= tr(X^T Y)$ denotes the trace scalar product.
The concave conjugate of $g$ is defined as
$$f_*(Y)=\inf_{X\in {\rm dom}(-f)}\{\langle X,Y\rangle -f(X) \}.$$
For a function $f(\cdot,\cdot)$ with two variables, $f^*(\cdot,\cdot)$ denotes the partial convex conjugate function with respect to the first variable. At
the same time, $f^*(\cdot;\cdot)$ denotes the convex conjugate function with respect to both two variables.
Similarly, we can define $f_*(\cdot,\cdot)$ and $f_*(\cdot;\cdot)$. The support function of the set $S\subset \mathbb{R}^{k\times k}$ is defined as $\delta^*(Y|S)=\sup_{X\in S} tr(X^T Y).$

The space of symmetric matrices of dimension $k$ is denoted by $\mathbb{S}^k$. Let $p\in  [1,+\infty]$ and $A \in \mathbb{S}^k$ be a real symmetric matrix. Then the Schatten norm $\|\cdot\|_{\sigma p}$ can be defined by
\begin{eqnarray*}
\|A\|_{\sigma p}\triangleq
\left\{ \begin{array}{ll}
\left( \sum_{i=1}^{k} \sigma_i^p(A)   \right)^{\frac{1}{p}},\, &\,\,\, 1\leq p< \infty,
\\
\sigma_{max}\left(A \right),\, & \,\,\,p=\infty. \end{array} \right.
\end{eqnarray*}
where $\sigma_i(A)$ is the absolute value of the $i$-largest eigenvalue of $A$ which has real eigenvalues.
Let $A,B \in \mathbb{S}^k$ and $p,q\in [1,+\infty]$ satisfy $\frac{1}{p}+\frac{1}{q}=1$. Then
$tr(A^T B)\leq \|A\|_{\sigma p}\|B\|_{\sigma q}$ (see \cite{B}). In particular,
$tr(A^T B)\leq  \sigma_{max}\left(A \right)\|B\|_{\sigma 1}$.

\begin{lemma}\cite{BT}\label{Fenchel}
Let $f,-g:X \rightarrow (-\infty,+\infty]$ be proper convex and lower-semicontinuous functions and
 $X$ a real Banach space. If there exists an element $x_0\in \mbox{dom}(f)\bigcap \mbox{dom}(g)$ such that either $f$
or $g$ is continuous at $x_0$, then the following equality holds:
$$\min\{f(x)-g(x),x\in X\}=\max\{g_*(y)-f^*(y)|y\in X^*\}, $$
where $dom(f)$ is the effective domain of the function $f$, $X^*$ is the dual space of $X$.
\end{lemma}

\section{Globalized distributionally robust counterpart (GDRC)}
\noindent \setcounter{equation}{0}
In distributionally robust optimization, the authors study a variant of
the stochastic constraint where the probability distribution belong to a given ambiguity information set.
In this paper, we focus on the following distributionally robust counterpart
\begin{eqnarray*}\label{DRO1}
&&E_P[g(x,\xi)]\leq 0,\;\forall P\in \mathcal{P},
\end{eqnarray*}
where $x\in \mathbb{R}^n$ is the decision variable, $\xi: \Omega\rightarrow \Xi \in\mathbb{R}^k$ is the random variable defined on probability space $(\Omega,\mathfrak{F},P)$, $g:\mathbb{R}^n\times \mathbb{R}^k\rightarrow \mathbb{R}$ is a real value function, and $\mathcal{P}$ is defined as the uncertainty distribution set which contains all the "physically possible" realizations. Obviously, the optimal solution of distributionally robust optimization would be conservative since the distribution takes the worst case scenario in the ambiguity information set. To reduce the conservatism, we introduce the following globalized distributionally robust counterpart (GDRC)
\begin{eqnarray}\label{gdrc1}
E_P[g(x,\xi)]\leq \inf_{P'\in \mathcal{P}'} \mathcal{H}(P,P'),\;\forall P\in \mathcal{P},
\end{eqnarray}
where $\mathcal{P}'$ is viewed as the inner uncertainty distribution set of ``normal range" of realizations and
 $\mathcal{H}(\cdot,\cdot)$ is a nonnegative distance-like function which satisfies
$$P=P'\Leftrightarrow \mathcal{H}(P,P')=0,\,
P\neq P'\Leftrightarrow \mathcal{H}(P,P')>0.$$
Clearly, $\mathcal{P}'\subset \mathcal{P}$. The term on the right-hand side of \eqref{gdrc1} represents the
allowable violation of the constraint. Generally, the magnitude of the allowable violation is correlated with the ``distance"
between $\mathcal{P}$ and the set $\mathcal{P}'$. The further $\mathcal{P}$
is far from $\mathcal{P}'$, the bigger allowable violation will be chosen.
Actually, the ``physically possible" distributions that further away from the inner set $\mathcal{P}'$ are less to happen in practice, but all of them should be taken into consideration.

%Ben-Tal et al. \cite{BBB} introduced the soft robust model of DRO with the purpose of reducing the conservatism. Xu et al. \cite{XCM1} presented that a probabilistic envelope constraint could be rewritten as a comprehensive robust constraint.

\iffalse
The distance function $H(\cdot,\cdot)$ in \cite{LK} which measure the distance between means and covariances has
the following form:
$$H\left((\mu,\Sigma),(\mu',\Sigma')\right)=\left\{\sum_{i=1}^{\mathcal{I}}c_i r_i(\gamma)\left|
\inf_{v\in U_{\mu}}\|\mu-v\|\leq \gamma_1,\,\inf_{\sigma\in U_{\Sigma}}\|\Sigma-\sigma\|\leq \gamma_2
\right.\right\}, $$
where $c_i\geq 0$ and $r_i$ are convex norm functions for $i=1,\cdots,\mathcal{I}$. The kernel of
\fi

With the purpose of providing a reasonably robust portfolio selection for investors in the presence of rare but high-impact realization of moment uncertainty, Li and Kwon \cite{LK} used a penalized moment-based optimization approach to study the portfolio optimization problem and rewrote it into semidefinite programs. In the following, we also  refine GDRC \eqref{gdrc1} into a moment-based framework so that the requirement of the exact distribution description which is hard to obtain can be neglected. In the new framework, we replace distance-like function between distribution measures with distance function between second order moment information and define the following GDRC
\begin{eqnarray}\label{GDRC001}
E_P[g(x,\xi)]\leq \!\!\inf_{(\mu',\Sigma')\in \mathcal{U}_{(\mu'\!,\Sigma')}}\!\!\!\!\!\!\! H\left((\mu,\Sigma),(\mu',\Sigma')\right),\,\forall P\in \mathcal{P},
\end{eqnarray}
where $\mu$ and $\Sigma$ are mean vector and covariance matrix of the distribution measure $P$, respectively. Obviously, the structure of the distance function $H(\cdot,\cdot)$ is a key issue to solve globalized distributionally robust constraint \eqref{GDRC001}.
In \cite{LK}, the authors introduced the following distance-like function
$$H\left((\mu,\Sigma),(\mu',\Sigma')\right)=\left\{\sum_{i=1}^{\mathcal{I}}k_i r_i(\gamma_1,\gamma_2)\left|
\inf_{v\in U_{\mu}}\|\mu-v\|\leq \gamma_1,\,\inf_{\sigma\in U_{\Sigma}}\|\Sigma-\sigma\|\leq \gamma_2
\right.\right\}, $$
where $k_i\geq 0$ and $r_i$ is a convex norm function for $i=1,\cdots,\mathcal{I}$. Based on two independent distance functions $\inf_{v\in U_{\mu}}\|\mu-v\|$ and $\inf_{\sigma\in U_{\Sigma}}\|\Sigma-\sigma\|$, the total penalty distance measure is implemented using convex penalty functions $r_i$ which adjusted the ambiguity of mean and covariance dependently.

Different from
the above composite structure of the distance function, we want to directly define the distance function with the ambiguity mean and covariance, i.e.,
\begin{eqnarray*}
{\rm(C1)} \qquad  H\left((\mu,\Sigma),(\mu',\Sigma')\right)=\varphi(\mu,\mu')+\psi(\Sigma,\Sigma'),
\end{eqnarray*}
where  $\varphi(\mu,\mu')$ and $\psi(\Sigma,\Sigma')$ are two nonnegative jointly convex functions with $\varphi(\mu,\mu)=0$ and $\psi(\Sigma,\Sigma)=0$, respectively. Since ${\rm(C1)}$ type distance function has a general structure about mean and covariance, we can have more flexible choices about measuring the distance between $(\mu,\Sigma)$ and $(\mu',\Sigma')$. For example, let $\varphi(\mu,\mu')=\sum_{i=1}^{k}\mu'_{i}p(\frac{\mu_i}{\mu'_{i}})$, where $p(\cdot)$ is referred to as the $\phi$-divergence function and $p(t)$ is convex for $t\ge 0$ with $p(1)=0$.

Since the distance-like function $H\big((\mu,\Sigma),(\mu',\Sigma')\big)$ does not have the separable structure about mean and covariance, it is difficult to obtain the deterministic reformulation of GDRC \eqref{GDRC001}. An alternative way is to discuss GDRC \eqref{GDRC001} with some special distance functions, i.e.,
$${\rm(C2)} \qquad H\left((\mu,\Sigma),(\mu',\Sigma')\right)=\eta \left(\mu-\mu_0\right)^T\Sigma^{-1}\left(\mu-\mu_0\right),$$
where $\eta>0$. We notice that $H(\cdot,\cdot)$ is jointly convex in $\mu$ and $\Sigma$.

Sometimes, the ${\rm(C1)}$  and ${\rm(C2)}$ type functions are not valid when the ambiguity information set only contains mean and the support information (see \cite{GNX,LSXZ,NMU}). In this case, the distance function can be modified as follows:
$${\rm(C3)} \qquad  H\left((\mu,\Sigma),(\mu',\Sigma')\right)=\varphi(\mu,\mu'), $$
where $\varphi(\mu,\mu')$ is a nonnegative jointly convex function with $\varphi(\mu,\mu)=0$.

In addition, the computational tractability of constraint \eqref{gdrc1} also depends on the structure of distribution set $\mathcal{P}$ and $\mathcal{U}_{(\mu',\Sigma')}$. Various type of distributionally robust optimization problems have been discussed under various uncertainty moments sets such as ellipsoidal moments sets \cite{DY}, interval moments sets \cite{NMU}, convex and compact moments sets \cite{DHX} and so on. In the following, we focus on the GDRC \eqref{gdrc1} under assumptions that the first and second order moments belong to corresponding convex sets with ${\rm(C1)}$  type distance function and ${\rm(C2)}$ type distance function, respectively. We also discuss the GDRC \eqref{gdrc1} under assumptions that the first order moment set and the support information set are both convex and compact with ${\rm(C3)}$ type distance function.

\section{ GDRC with ${\rm(C1)}$ type function under second order moment information}
Generally speaking, the uncertainty set about mean vector and covariance matrix is separable in most cases \cite{DY,MZDR,NMU}.
In this section, we assume that the mean vector and covariance matrix for the random return $\xi$ belong to different convex sets, respectively. Thus, the ambiguity information set has the following form
\begin{eqnarray}\label{P}
\mathcal{P}_1
=\left\{
P\in
\mathfrak{P}: \mu\in \mathcal{U}_{\mu1},\,
\Sigma\in \mathcal{U}_{\Sigma1}
\right\},
\end{eqnarray}
where $\mathfrak{P}$ denotes the set of all probability measures on space $(\mathbb{R}^k,\mathfrak{B})$
with the Borel $\sigma$-algebra $\mathfrak{B}$ on $\mathbb{R}^k$,  $\mathcal{U}_{\mu1}$ and $\mathcal{U}_{\Sigma1}$, two convex sets, represent the ``physically possible" realizations for mean and covariance, respectively. Similarly, we assume that
\begin{eqnarray*}
\mathcal{U}_{(\mu',\Sigma')}=\left\{(\mu',\Sigma'):\mu'\in\mathcal{U}_{\mu2},\Sigma'\in\mathcal{U}_{\Sigma2}\right\},
\end{eqnarray*}
where $\mathcal{U}_{\mu2}$ and $\mathcal{U}_{\Sigma2}$, two compact and convex sets, denote the ``normal range" of realizations for mean and covariance, respectively.  Hence, the GDRC in terms of the moment-based framework with ${\rm(C1)}$ type distance function can be reformulated as
\begin{eqnarray}\label{gdrc1+1}
E_{P\thicksim(\mu,\Sigma)}[g(x,\xi)]\leq
\min_{\mu'\in \mathcal{U}_{\mu2}} \varphi(\mu,\mu')
+\min_{\Sigma'\in \mathcal{U}_{\Sigma2}} \psi(\Sigma,\Sigma')
,\,\forall P\in \mathcal{P}_1.
\end{eqnarray}

If $\mu\in \mathcal{U}_{\mu2}$ and $\Sigma\in \mathcal{U}_{\Sigma2}$, then $\min_{\mu'\in \mathcal{U}_{\mu2}} \varphi(\mu,\mu')=0$,  $\min_{\Sigma'\in \mathcal{U}_{\Sigma2}} \psi(\Sigma,\Sigma')=0$ and so the globalized distributionally robust constraint can be simplified to the original constraint $E_{P}[g(x,\xi)]\leq 0, \,\forall P\in \mathcal{P}_1$.  If $\mu\in \mathcal{U}_{\mu1} \backslash\mathcal{U}_{\mu2}$ and $\Sigma\in \mathcal{U}_{\Sigma1}\backslash \mathcal{U}_{\Sigma2}$, then the allowable violation of \eqref{gdrc1+1} is controlled by the sum of the distance of $\mu$ to $\mathcal{U}_{\mu2}$ and the distance of $\Sigma$ to $\mathcal{U}_{\Sigma2}$.

In practise, we can estimate the uncertain mean $\mu$ and covariance $\Sigma$ from historical data. Denote the empirical mean and covariance by $\mu_0$ and $\Sigma_0$. In general, the distance between the exact first (resp. second) order moment $\mu$ (resp., $\Sigma$) and the empirical first (resp., second) order moment $\mu_0$ (resp., $\Sigma_0$) is not large. Then, the uncertain sets $\mathcal{U}_{\mu i}$ and $\mathcal{U}_{\Sigma i}$ can be defined as follows:
\begin{eqnarray}\label{+dis}
\mathcal{U}_{\mu i}=\left\{\mu=\mu_0+A\zeta\, |\,\zeta\in U_i \right\},\,\mathcal{U}_{\Sigma i}=\left\{\Sigma=\Sigma_0+\Xi\, |\,\Xi\in Z_i \right\},\,i=1,2,
\end{eqnarray}
where $\zeta$ and $\Xi$ are mentioned as primitive uncertainties, $A\in \mathbb{S}^k$, $U_1$ and $Z_1$ are convex, and $U_2$ and $Z_2$ are compact and convex sets with $0$ element. Here we assume that $Z_2\subset Z_1$ and $U_2\subset U_1$.  Thus, one can see that $\mathcal{U}_{\Sigma 2}\subset  \mathcal{U}_{\Sigma 1}$ immediately.

In order to obtain the reformulation of problem \eqref{gdrc1+1},  we assume that the constraint function has the following piecewise-linear form
$$g(x,\xi)=\max_{i=1,2,\cdots,m}\{a_i(x)+b_iw^T\!(x)\xi\},$$
where $a_i(x)$ and $w(x)$ are affine functions. Usually this general piecewise linear function $g(x,\xi)$ can be extended to the popular CVaR risk measure and a more general optimized certainty equivalent risk measure \cite{NMU}. Moreover, this structure also can be transformed into a piecewise linear utility function
$u(x,\xi):=\min_{i=1,2,\cdots,m}\left\{c_i+d_ix^T\xi  \right\}$ with the decision variable $x$ and the return vector $\xi$ when we set $a_i(\cdot)=-c_i$ and $b_i=-d_i$ (see \cite{LK}). Next, we will show that GDRC \eqref{gdrc1+1} under information set \eqref{P} with ${\rm(C1)}$ type distance function can be reformulated as a deterministic equivalent system of inequalities.
\begin{thm}\label{gdrcth1}
GDRC \eqref{gdrc1+1} under information set \eqref{P} can be reformulated as the following system of inequalities:
\begin{eqnarray}\label{th1gdr}
\left\{
\begin{array}{l}
p+s\le 0,\,t^2-4zv\leq 0,\, z>0,\\
tr\left(Q^T\Sigma_0\right)
+\delta^*\left(\left(Q-Y\!\right)|Z_{1}\right)
+\delta^*\left(Y|Z_{2}\right)+\psi^*\left(Y;-Y\right)+v-s
\leq 0,\\
\delta^*\big(A\left(b_iw(x)\!-\!\lambda_i\!\right)|U_{1}\big)+\delta^*\left(A\lambda_i|U_{2}\right) +\varphi^*\left(\lambda_i;-\lambda_i\right)\\
\qquad\qquad\qquad+ b_iw^T\!(x)\mu_0
-p+a_i(x)-b_it+b_i^2 z\leq 0 ,\,\forall\, i=1,2,\cdots, m,\\
\left(
\begin{array}{cc}
4z &  w^T\!(x)\\
w(x)&Q \\
\end{array} \right)\succeq 0,
\end{array}
\right.
\end{eqnarray}
where $p,s,z,t,v\in \mathbb{R}, x\in \mathbb{R}^n$, $\lambda_i\in \mathbb{R}^k$, and $Y,Q\in \mathbb{R}^{k\times k}$ are decision variables for $i=1,2,\cdots,m$. Moreover, the feasible set of system \eqref{th1gdr} is convex.
\end{thm}

\begin{proof}
Rewriting semi-infinite constraint \eqref{gdrc1+1}, one has
\begin{eqnarray}\label{gdrc2}
\sup_{P\in \mathcal{P}_1,\mu'\in \mathcal{U}_{\mu2},\Sigma'\in \mathcal{U}_{\Sigma2}}
\left\{
  E_{P\thicksim(\mu,\Sigma)}\left[\max_{i=1,2,\cdots,m}\{a_i(x)+b_iw^T\!(x)\xi\}\right]-\varphi(\mu,\mu')
- \psi(\Sigma,\Sigma')\right\}\leq 0.
\end{eqnarray}
We construct the following auxiliary information set
\begin{eqnarray*}
\mathcal{P}_{11}\left(\mathbb{R}^k,{\mu},{\Sigma}\right)=\left\{
P\in
\mathfrak{P}: P({\xi}\in
\mathbb{R}^k)=1,\, E[{\xi}]={\mu},\,E[({\xi}-{E}[{\xi}])({\xi}-{E}[{\xi}])^T]={\Sigma}
\right\},
\end{eqnarray*}
where the mean vector $\mu$ and covariance matrix $\Sigma$ are given.  Then, we can reformulate the left-hand side term of inequality \eqref{gdrc2} as
\begin{eqnarray}\label{supsup}
\sup_{
\mu\in \mathcal{U}_{\mu1},
\Sigma\in \mathcal{U}_{\Sigma1},
\mu'\in \mathcal{U}_{\mu2},\Sigma'\in \mathcal{U}_{\Sigma2}
}\!\!
\left\{
\sup_{{P}\in
\mathcal{P}_{11}}\!\!
E_{P}\!\!\left[\max_{i=1,2,\cdots,m}\left\{a_i(x)+b_iw^T\!(x)\xi\right\}\right]
-\varphi(\mu,\mu')- \psi(\Sigma,\Sigma')\right\}.
\end{eqnarray}
Let $y=w^T\!(x)\xi$. Clearly, the mean and variance of $y$ are $w^T\!(x){\mu}$
and $w^T\!(x){\Sigma}\,w(x)$, respectively. Thus, the inner subordinate expectation problem in \eqref{supsup} under
information set $\mathcal{P}_{11}$ can be unfolded as follows:
\begin{eqnarray}\label{org2+}
&&\sup_{\zeta\in \mathcal{M}_+}\;
\int_{\mathbb{R}}\max_{i=1,2,\cdots,m}\left\{a_i(x)+b_iy\right\}\zeta(dy)\\
&&\;\;\mbox{s.t.}\;\; \int_{\mathbb{R}}\zeta(dy)=1,\; \int_{\mathbb{R}} y\,\zeta(dy)=w^T\!(x){\mu},
\int_{\mathbb{R}} y^2\zeta(dy)=w^T\!(x){\Sigma}\,w(x)+(w^T\!(x){\mu})^2,\nonumber
\end{eqnarray}
where $\mathcal{M}_+$ represents the cone of nonnegative Borel measures on $\mathbb{R}$ and $\zeta(\cdot)$ is the decision variable. The dual of problem \eqref{org2+} can be rewritten as follows:
\begin{eqnarray}\label{org1}
&&\inf_{\gamma_0,\gamma_1,\gamma_2}
\gamma_0+\gamma_1w^T\!(x){\mu}+\gamma_2\left(w^T\!(x){\Sigma}\,w(x)+(w^T\!(x){\mu})^2\right)\\\nonumber
&&\quad\,\mbox{s.t.}\quad \gamma_0-a_i(x)+(\gamma_1-b_i)y+\gamma_2y^2\geq 0,\; \forall y\in \mathbb{R},\; i=1,2,\cdots,m,
\end{eqnarray}
where $\gamma_0,\gamma_1,\gamma_2\in \mathbb{R}$ are the dual variables for the constraints respectively.   It is easy to verify that the feasible region is nonempty when $\gamma_2>0$ and the minimum value in the left side is attained at $y^*=\frac{b_i-\gamma_1}{2\gamma_2}$, for $i=1,2,\cdots,m$. Therefore, problem (\ref{org1}) can be reformulated as
\begin{eqnarray}\label{org2}
&&\inf_{\gamma_0,\gamma_1,\gamma_2}
\gamma_0+\gamma_1w^T\!(x){\mu}+\gamma_2\left(w^T\!(x){\Sigma}\,w(x)+
\left(w^T\!(x){\mu}\right)^2\right)\\\nonumber
&&\quad\mbox{s.t.}\;\; \gamma_0-a_i(x)-\frac{(\gamma_1-b_i)^2}{4\gamma_2}\geq 0, \; i=1,2,\cdots,m.
\end{eqnarray}
Suppose $z>0$ and let
$$
\gamma_0=p+\frac{1}{4z}{(t-{w}^T({x}){\mu})^2}, \;\; \gamma_1=\frac{1}{2z}(t-{w}^T({x}){\mu}),\;\; \gamma_2=\frac{1}{4z}.
$$
Then, problem (\ref{org2}) is equivalent to
\begin{eqnarray*}\label{org3}
&&\inf_{p,t,z,s}\;\; p+s\\\nonumber
&&\;\mbox{s.t.}\quad\; \frac{w^T\!(x){\Sigma}\,w(x)+t^2}{4z}-s \leq 0,\;z> 0,\\\nonumber
&&\quad\quad\;\;\; b_iw^T\!(x){\mu}-p+a_i(x)-b_it+b_i^2 z\leq 0,\,i=1,2,\cdots, m
\end{eqnarray*}
and so problem \eqref{supsup} becomes
\begin{eqnarray}\label{+obje}
&&\sup_{\mu\in \mathcal{U}_{\mu1},
\Sigma\in \mathcal{U}_{\Sigma1},\mu'\in \mathcal{U}_{\mu2},\Sigma'\in \mathcal{U}_{\Sigma2}}\inf_{p,t,z,s}\;\; p+s-\varphi(\mu,\mu')-\psi(\Sigma,\Sigma')\\\nonumber
&&\;\mbox{s.t.}\quad\; \frac{w^T\!(x){\Sigma}\,w(x)+t^2}{4z}-s \leq 0,\;z> 0,\\\nonumber
&&\quad\quad\;\;\;b_iw^T\!(x){\mu}-p+a_i(x)-b_it+b_i^2 z\leq 0 ,\,i=1,2,\cdots, m.
\end{eqnarray}
When the uncertainty variables have been dispersed in several constraints function, the optimal solution become conservative.
Next we try to collect $\mu$ and $\Sigma$ in one constraint, respectively. By eliminating $p$ and $s$, we transform problem \eqref{+obje} into
\begin{eqnarray*}\label{+obje2}
&&\sup_{\scriptstyle \mu,\mu' \atop \scriptstyle \Sigma,\Sigma'}  \inf_{t,z>0}\, \max_{i=1,2,\cdots, m}\left\{ b_iw^T\!(x){\mu}+a_i(x)-b_it+b_i^2 z\right\}-\varphi(\mu,\mu')+\frac{w^T\!(x){\Sigma}\,w(x)+t^2}{4z}-\psi(\Sigma,\Sigma').
\end{eqnarray*}
By introducing new variables $p$ and $s$ again, we immediately reformulate the above problem as follows:
\begin{eqnarray}\label{+max-min}
&&\sup_{\mu\in \mathcal{U}_{\mu1},
\Sigma\in \mathcal{U}_{\Sigma1},\mu'\in \mathcal{U}_{\mu2},\Sigma'\in \mathcal{U}_{\Sigma2}}\inf_{p,t,z,s}\;\; p+s\\ \nonumber
&&\;\mbox{s.t.}\quad\; \frac{w^T\!(x){\Sigma}\,w(x)}{4z}-\psi(\Sigma,\Sigma')+\frac{t^2}{4z}-s \leq 0,\;z> 0,\\\nonumber
&&\quad\quad\;\;\;b_iw^T\!(x){\mu}-\varphi(\mu,\mu')-p+a_i(x)-b_it+b_i^2 z\leq 0 ,\,i=1,2,\cdots, m.
\end{eqnarray}
In what follows, we will prove that problem \eqref{+max-min} is equivalent to the following one
\begin{eqnarray}\label{+max-min2+}
&&\inf_{p,t,z,s}\; p+s \\
&&\;\mbox{s.t.}\;\max_{\Sigma\in \mathcal{U}_{\Sigma1},\Sigma'\in \mathcal{U}_{\Sigma2}} \nonumber \left\{\frac{w^T\!(x){\Sigma}\,w(x)}{4z}-\psi(\Sigma,\Sigma')\right\}+\frac{t^2}{4z}-s \leq 0,\;z> 0,\\\nonumber
&&\quad\;\;\; \max_{\mu\in \mathcal{U}_{\mu1},\mu'\in \mathcal{U}_{\mu2}}\!\! \left\{b_iw^T\!(x){\mu}\!-\!\varphi(\mu,\mu')\right\}\!-\!p\!+\!a_i(x)\!-\!b_it\!+\!b_i^2 z\!\leq 0 ,\,i=1,2,\cdots, m.
\end{eqnarray}

For fixed $x$, suppose that $(\mu^*,\mu'^*,\Sigma^*,\Sigma'^*,p^*,t^*,z^*,s^*)$ is the optimal solution of \eqref{+max-min}. We also assume that, for some $\mu_1\in \mathcal{U}_{\mu1},
\Sigma_1\in \mathcal{U}_{\Sigma1},
\mu'_1\in \mathcal{U}_{\mu2},\Sigma'_1\in \mathcal{U}_{\Sigma2}$, at least one of the constraints is violated for $(p^*,t^*,z^*,s^*)$, which means one of the following inequalities holds
\begin{eqnarray*}
&& s^*< \frac{(t^*)^2}{4z^*} +\frac{{w}({x})^T{\Sigma_1}{w}({x})}{4z^*}-\psi(\Sigma_1,\Sigma'_1), \\
&& p^*<b_iw^T\!(x){\mu_1}-\varphi(\mu_1,\mu'_1)+a_i(x)-b_it^*+b_i^2 z^*.
\end{eqnarray*}
Suppose that the first inequality holds.  Then we can obtain $s_1=\frac{(t^*)^2}{4z^*} +\frac{{w}({x})^T{\Sigma_1}{w}({x})}{4z^*}-\psi(\Sigma_1,\Sigma'_1)$ such that $(\mu_1,\mu'_1,\Sigma_1,\Sigma'_1,p^*,t^*,z^*,s_1)$ is feasible for \eqref{+max-min}. It follows that
\begin{eqnarray*}
s_1+p^*=\left(\frac{(t^*)^2}{4z^*} +\frac{{w}({x})^T{\Sigma_1}{w}({x})}{4z^*}-\psi(\Sigma_1,\Sigma'_1)+p^*\right)>s^*+p^*,
\end{eqnarray*}
which is a contradiction with the fact that $(\mu^*,\mu'^*,\Sigma^*,\Sigma'^*,p^*,t^*,z^*,s^*)$ is the optimal solution of \eqref{+max-min}. Thus, we have
\begin{eqnarray*}
&& s^*\ge \frac{(t^*)^2}{4z^*} +\frac{{w}({x})^T{\Sigma}{w}({x})}{4z^*}-\psi(\Sigma,\Sigma'),\;\; \forall\, \Sigma\in \mathcal{U}_{\Sigma1},\,\Sigma'\in \mathcal{U}_{\Sigma2}.
\end{eqnarray*}
Similarly, we can prove that
\begin{eqnarray*}
b_iw^T\!(x){\mu}-\varphi(\mu,\mu')-p^*+a_i(x)-b_it^*+b_i^2 z^*\leq 0 ,\,\forall\,\mu\in \mathcal{U}_{\mu1},\,\mu'\in \mathcal{U}_{\mu2},\,i=1,2,\cdots, m.
\end{eqnarray*}
In conclusion, we obtain that problem \eqref{+max-min} is equivalent to problem \eqref{+max-min2+}.
Moreover, by $\Sigma \succeq 0$ and  Lemma 8.4.12 in \cite{B}, \eqref{+max-min2+} can be rewritten as follows:
\begin{eqnarray}\label{+max-min2}
&&\inf_{p,t,z,s,v}\; p+s \\\label{cd1}
&&\;\mbox{s.t.}\;\max_{\Sigma\in \mathcal{U}_{\Sigma1},\Sigma'\in \mathcal{U}_{\Sigma2}} \left\{tr(Q^T\Sigma)-\psi(\Sigma,\Sigma')\right\}+v-s \leq 0,\\\label{cd2}
&&\quad\;\;\; \max_{\mu\in \mathcal{U}_{\mu1},\mu'\in \mathcal{U}_{\mu2}}\!\! \left\{b_iw^T\!(x){\mu}\!-\!\varphi(\mu,\mu')\right\}\!-\!p\!+\!a_i(x)\!-\!b_it\!+\!b_i^2 z\!\leq 0 ,\,\forall\, i=1,2,\cdots, m,\\
&&\quad\;\;\;  Q-\frac{w(x)w^T\!(x)}{4z} \succeq 0,\,t^2-4zv\leq 0, \,z>0.\nonumber
\end{eqnarray}
Next, we focus on the  reformulation of problem \eqref{+max-min2}. By Schur's complement, the positive semidefinite constraint in problem \eqref{+max-min2} can be written as
$$\left(
\begin{array}{cc}
4z &  w^T\!(x)\\
w(x)&Q \\
\end{array} \right)\succeq 0. $$
By introducing new matrices $M_1$ and $M_2$, we reformulate the subordinate maximization problem in constraint \eqref{cd1}  as
\begin{eqnarray*}
&&
\max_{\Sigma\in \mathcal{U}_{\Sigma1},\Sigma'\in \mathcal{U}_{\Sigma2},M_1,M_2}
\left\{tr(Q^T\Sigma)-\psi(M_1,M_2)\,|\,\Sigma=M_1,\Sigma'=M_2\right\}
\nonumber\\
&\!=\!\!\!&\max_{\Sigma\in \mathcal{U}_{\Sigma1},\Sigma'\in \mathcal{U}_{\Sigma2},M_1,M_2}\min_{Y,Y_1}
\left\{tr(Q^T\Sigma)\!-\!\psi(M_1,M_2)
\!+\!tr\!\left(Y_1^T(\Sigma\!-\!M_1)\right)\!+\!tr\!\left(Y^T(\Sigma'\!-\!M_2)\right)\right\}
\nonumber\\
&\!=\!\!\!&\min_{Y,Y_1} \max_{\Sigma\in \mathcal{U}_{\Sigma1},\Sigma'\in \mathcal{U}_{\Sigma2},M_1,M_2}
\left\{tr\!\left(\!\left(Q\!+\!Y_1\right)^T{\Sigma}\!\right)
\!+\!tr\!\left(Y^T\Sigma'\right)
-tr\!\left(Y_1^TM_1\right)\!-\!tr\!\left(Y^TM_2\right)\!-\!\psi(M_1,M_2)\right\}
\\
&\!=\!\!&\!\!\min_{Y,Y_1}\Bigg\{
\underbrace{
\max_{\Sigma\in \mathcal{U}_{\Sigma1}}\!
\left\{tr\!\left(\left(Q+Y_1\right)^T{\Sigma}\!\right)\right\}}
_{h_1\left(Q\!+\!Y_1\! \right)}
\!+\!
\underbrace{\max_{\Sigma'\in \mathcal{U}_{\Sigma2}}\!\!\left\{tr\!\left(Y^T\Sigma'\right)\!\right\}}_
{h_2\left(Y \right)}\\
&&\quad\quad\quad\quad\quad\quad\quad\quad\quad\quad\quad\!+\!
\underbrace{
\max_{M_1,M_2}\!\left\{\!-\!tr\!\left(Y_1^TM_1\right)\!-\!tr\!\left(Y^TM_2\right)\!-\!\psi(M_1,M_2)\right\}}
_{h_3\left(Y,Y_1 \right)}
\Bigg\}.
\end{eqnarray*}
Reformulating the three maximum problems, we have
\begin{eqnarray*}
&&h_1\left(Q+Y_1 \right)=
\max_{\Xi\in Z_{1}}\!
\left\{ tr\left( \left(Q+Y_1\right)^T\!\Sigma_0\right)
+ tr\left(\left(Q+Y_1\right)^T\Xi\right)
\right\}\\
&&\,\quad\qquad\qquad=tr\left( \left(Q+Y_1\right)^T\Sigma_0\right)\!
+\delta^*\left(\left(Q+Y_1\right)|Z_{1}\right),\\
&&h_2\left(Y \right)=
\max_{\Xi\in Z_{2}}\!\!\left\{tr\left(Y^T\Sigma_0\right)+tr\left(Y^T\Xi\right)\,\right\}
=tr(Y^T\Sigma_0)+\delta^*\left(Y|Z_{2}\right),\\
&&h_3\left(Y,Y_1\right)= \max_{M_1,M_2}\!\left\{\!-\!tr\!\left(Y_1^TM_1\right)\!-\!tr\!\left(Y^TM_2\right)\!-\!\psi(M_1,M_2)\right\}
=\psi^*\left(-Y_1;-Y\right).
\end{eqnarray*}
For the third maximum problem, one has
\begin{eqnarray*}
\psi^*\left(-Y_1;-Y\right)\!\!&\geq &\!\!\!\!
\max_{M_1=M_2}\left\{-tr\left(\left(Y_1+Y\right)^TM_1\right)\!-\!\psi(M_1,M_1)\right\}\\
\!\!&=&\!\!\!\max_{M_1} tr\left(-(Y_1+Y)^TM_1\right)\\
\!\!&=&\!\!\!\left\{ \begin{array}{ll} 0,\, & \mbox{if}\, Y_1=-Y,\\ +\infty,\, & \mbox{if} \, Y_1\neq -Y. \end{array} \right.
\end{eqnarray*}
Then we have $Y_1=-Y$ and so problem \eqref{cd1} can be reformulated as
\begin{eqnarray*}
&&
tr\left(Q^T\Sigma_0\right)
+\delta^*\left(\left(Q-Y\!\right)|Z_{1}\right)
+\delta^*\left(Y|Z_{2}\right)+\psi^*\left(Y;-Y\right)+v-s
\leq 0,
\end{eqnarray*}
where$\,Y\in \mathbb{R}^{k\times k} \,$ is the decision variable.

Similarly, the subordinate maximization problem in \eqref{cd2} can be rewritten as follows:
\begin{eqnarray*}
&&\delta^*\big(A\left(b_iw(x)\!-\!\lambda_i\!\right)|U_{1}\big)+\delta^*\left(A\lambda_i|U_{2}\right) +\varphi^*\left(\lambda_i;-\lambda_i\right)\\
&&\qquad\qquad\qquad\quad
+ b_iw^T\!(x)\mu_0
-p+a_i(x)-b_it+b_i^2 z\leq 0 ,\,\forall\,i=1,2,\cdots, m,
\end{eqnarray*}
where $p,z,t$ and $\lambda_i$ are decision variables for $i=1,2,\cdots, m.$
Based on the above discussion, it completes the proof.
\end{proof}

From Theorem \ref{gdrcth1}, we notice that the globalized distributionally robust optimization  with ${\rm(C1)}$ type function under information set \eqref{P} can be rewritten as a convex programming problem. It is easy to see that the computation of system \eqref{th1gdr} involving $\mu$ and $\Sigma$ are all separable, i.e.,  the computation about $\varphi(\cdot;\cdot), \psi(\cdot;\cdot)$ and $\delta^*(\cdot|U_1), \delta^*(\cdot|U_2), \delta^*(\cdot|Z_1), \delta^*(\cdot|Z_2)$ are all independent. Ben-tal et al. \cite{BBHV} presented the globalized robust counterpart can be transformed in a computationally tractable system of inequalities for several choices of functions $\varphi(\cdot;\cdot)$ and convex sets $U_{i}$. For example, they use the distance measure $\varphi$ for two vectors $a$ and $a'$ based on the norm $\|a-a'\|$ and define $\varphi(a,a')=\alpha(\|a-a'\|)$ with $\alpha(\cdot)$ being a nonnegative convex function and $\alpha(0)=0$, they obtain $\varphi^*\left(\lambda;-\lambda\right)=\alpha^*(\|\lambda\|^*)$, where $\|\cdot\|^*$ is the dual norm of $\|\cdot\|$.
More examples about support function and conjugate function on $\mathbb{R}^{k}$ can be found in \cite{BBHV2,BBHV}.
For matrix $\Xi$, if we define
\begin{eqnarray}\label{th1vec+}
Z_{i}\triangleq Z_i(vec(\Xi)), \quad \psi(\Sigma,\Sigma')\triangleq\psi(vec(\Sigma),vec(\Sigma')),\,i=1,2,
\end{eqnarray}
then we can deal with support functions and conjugate functions matrix by employing the methods and results about vector.
By the similar discussion, the reformulation of GDRC \eqref{gdrc1+1} with the case \eqref{th1vec+} can be modified as
\begin{eqnarray}\label{th1gdr+vec}
\left\{
\begin{array}{l}
p+s\le 0,\,t^2-4zv\leq 0,\, z>0,\\
tr\left(Q^T\Sigma_0\right)
+\delta^*\left(\left(vec(Q)-y\!\right)|Z_{1}\right)
+\delta^*\left(y|Z_{2}\right)+\psi^*\left(y;-y\right)+v-s
\leq 0,\\
\delta^*\big(A\left(b_iw(x)\!-\!\lambda_i\!\right)|U_{1}\big)+\delta^*\left(A\lambda_i|U_{2}\right) +\varphi^*\left(\lambda_i;-\lambda_i\right)\\
\qquad\qquad\qquad+ b_iw^T\!(x)\mu_0
-p+a_i(x)-b_it+b_i^2 z\leq 0 ,\,\forall\,i=1,2,\cdots, m,\\
\left(
\begin{array}{cc}
4z &  w^T\!(x)\\
w(x)&Q \\
\end{array} \right)\succeq 0,
\end{array}
\right.
\end{eqnarray}
where $p,s,z,t,v\in \mathbb{R}, x\in \mathbb{R}^n, \lambda_i\in \mathbb{R}^k$, $y \in \mathbb{R}^{k^2}$ and $Q\in \mathbb{R}^{k\times k}$ are decision variables for $i=1,2,\cdots,m$. Actually, it is not difficult to find that the feasible set of \eqref{th1gdr+vec} is also convex. Here we show a computationally tractable example.
\begin{example} Let
\begin{eqnarray*}
&&U_{i}=\left\{\zeta\in \mathbb{R}^k| \|\zeta\|_{1}\leq \theta_i\right\},\,
Z_i=\left\{\Xi\in \mathbb{S}^k|\,\|vec(\Xi)\|_2\le \rho_i \right\},i=1,2,\\
&&\varphi(\mu,\mu')=\beta_1 \|\mu-\mu' \|_{1},\,   \psi(\Sigma,\Sigma')=\frac{\beta_2}{2}\|vec(\Sigma)-vec(\Sigma')\|^2_{2},
\end{eqnarray*}
with $\theta_1\geq \theta_2\geq 0$, $\rho_1\geq \rho_2\geq 0$ and $\beta_1>0,\beta_2>0.$   Then, the support functions and the conjugate functions can be given as follows:
\begin{eqnarray*}
&&\delta^*\big(A\left(b_iw(x)\!-\!\lambda_i\!\right)|U_{1}\big)=\theta_1\|A\left(b_iw(x)\!-\!\lambda_i\!\right)\|_{\infty},\\
&&\delta^*\left(A\lambda_i|{U}_{2}\right)=\theta_2\|A\lambda_i\|_{\infty},\,\forall\, i=1,2,\cdots,m,\\
&&\delta^*\!\left(vec(Q)-y |Z_{1}\!\right)=\|vec(Q)-y\|_2,\\
&&\delta^*\!\left(y|Z_{2}\right)\!=\|y\|_2,\\
&&\varphi^*(\lambda_i;-\lambda_i)=\left\{ \begin{array}{ll}\!\! 0, & \mbox{if} \; \|\lambda_i\|_{\infty}\leq \beta_1,\\ \!\!\infty, & \mbox{otherwise} \,. \end{array} \right.
\\
&&\psi^*\left(y;-y\right)=\frac{1}{2\beta_2} \|y\|^2_2.
\end{eqnarray*}
In this case, \eqref{th1gdr+vec} can be unfolded as
\begin{eqnarray*}
\left\{
\begin{array}{l}
p+s\le 0,\,t^2-4zv\leq 0,\, z>0,\, \|\lambda_i\|_{\infty}\leq \beta_1,\\
tr\left(Q^T\Sigma_0\right)+\|vec(Q)-y\|_2+\|y\|_2
+\frac{1}{2\beta_2} \|y\|^2_2+v-s \leq 0,\\
\theta_1\|A\left(b_iw(x)\!-\!\lambda_i\!\right)\|_{\infty}+\theta_2\|A\lambda_i\|_{\infty} + b_iw^T\!(x)\mu_0
-p+a_i(x)-b_it+b_i^2 z\leq 0 ,\,\forall\, i=1,2,\cdots, m,\\
\left(
\begin{array}{cc}
4z &  w^T\!(x)\\
w(x)&Q \\
\end{array} \right)\succeq 0,
\end{array}
\right.
\end{eqnarray*}
where $p,s,z,t,v\in \mathbb{R}, x\in \mathbb{R}^n, \lambda_i\in \mathbb{R}^k$, $y \in \mathbb{R}^{k^2}$ and $Q\in \mathbb{R}^{k\times k}$ are decision variables for $i=1,2,\cdots,m$.
\end{example}

However, since the structure of the matrix may be destroyed when the matrix has been transformed into the vector form, it is
worth considering the conjugate and support functions on $\mathbb{R}^{k\times k}$ directly. Some computationally tractable examples for support functions on $\mathbb{R}^{k\times k}$ can be found in \cite{DHX}. On the other hand, it is difficult to obtain the computationally tractable forms for conjugate function on $\mathbb{R}^{k\times k}$. Next we will show two special conjugate functions on $\mathbb{R}^{k\times k}$ and the corresponding reformulations \eqref{th1gdr} can be unfolded in a computationally tractable way.

\begin{example}
Let
\begin{eqnarray*}
&&U_{i}=\left\{\zeta\in \mathbb{R}^k|\,\|\zeta\|_{p}\leq \rho_i \right\},\;i=1,2,\\
&&Z_1=\left\{\Xi\in \mathbb{S}^k|\,0\preceq\Xi\preceq \theta_1\Xi_0 \right\},\,
Z_2=\left\{\Xi\in \mathbb{S}^k|\,0\preceq\Xi\preceq \theta_2\Xi_0, tr(\Xi D \Xi)\leq \tau \right\},\\
&&\varphi(\mu,\mu')=\frac{\beta_1}{2}(\mu-\mu')^T\Sigma^{-1}_0(\mu-\mu'),\,
\psi(\Sigma,\Sigma')=tr\left(\left(\Sigma-\Sigma'\right)^T\!\!P_1\left(\Sigma-\Sigma'\right)P_2 \right),
\end{eqnarray*}
where $\beta_1> 0, \tau\geq 0,\,\rho_1\geq\rho_2\geq 0, \theta_1\geq \theta_2\geq 0, p\geq 1$ and $\Xi_0\succeq0,D\succ 0, P_1\succ 0, P_2\succ 0$.
The support functions for $U_i$ and $Z_i$ are as follows:
\begin{eqnarray*}
&&\delta^*\big(A\left(b_iw(x)\!-\!\lambda_i\!\right)|U_{1}\big)=\rho_1\|A\left(b_iw(x)\!-\!\lambda_i\!\right)\|_{q},\\
&&\delta^*\left(A\lambda_i|{U}_{2}\right)=\rho_2\|A\lambda_i\|_{q},\,i=1,2,\cdots,m,\\
&&\delta^*\!\left(\!\left(Q-Y\right)\Big|Z_{1}\!\right)=
\min_{H_1}\left\{\theta_1 tr(H_1^T\,\Xi_0)\,\Big|H_1-Q+Y\succeq 0,H_1\succeq0 \right\},  \\
&&\delta^*\!\left(Y|Z_{2}\right)\!=\!\!\!\min_{V_1,V_2,H_2}\!\!\!\left\{\theta_2 tr(H_2^T\,\Xi_0)+\sqrt{\tau}\!\sqrt{tr(V_2D^{-1}V_2)}\,\big|\,V_1+V_2=Y,\,H_2-V_1\succeq 0,H_2\succeq0\right\},
\end{eqnarray*}
where $\|\cdot\|_{q}$ is the dual norm of $\|\cdot\|_{p}$ with $\frac{1}{p}+\frac{1}{q}=1$.  Furthermore, one has
\begin{eqnarray*}
\psi^*\left(Y;-Y\right)=\frac{1}{4}tr(Y^TP_1^{-1}YP_2^{-1}).
\end{eqnarray*}
Taking $P_1=\beta_2 E$ and $P_2=E$, we have $\psi(\Sigma,\Sigma')=\beta_2\|\Sigma-\Sigma'\|^2_{F}$
and so $\psi^*\left(Y;-Y\right)=\frac{1}{4\beta_2}tr(Y^T\!Y).$
Moreover, it is easy to check that
$$
\varphi^*\left(\lambda_i;-\lambda_i\right)=\frac{1}{2\beta_1}\lambda^T_i\Sigma_0 \lambda_i,\,\forall\, i=1,2,\cdots, m.
$$
In summarise, GDRC \eqref{th1gdr} can be rewritten as
\begin{eqnarray*}
\left\{
\begin{array}{l}
p+s\le 0,\,t^2-4zv\leq 0,\, z>0,\\
V_1+V_2=Y,\,H_2-V_1\succeq 0,\,H_1-Q+Y\succeq 0,\,H_1\succeq0,\,H_2\succeq0,\\
tr\left(Q^T\Sigma_0\right)
+tr\left(\left(\theta_1H_1+\theta_2 H_2\right)^T\!\Xi_0\right)\!+\!\sqrt{\tau}\!\sqrt{tr(V_2D^{-1}V_2)}+\frac{1}{4}tr(YP_1^{-1}YP_2^{-1})+v-s
\leq 0,\\
\rho_1\|A\left(b_iw(x)\!-\!\lambda_i\!\right)\|_{q}+\rho_2\|A\lambda_i\|_{q} +\frac{1}{2\beta_1}\lambda^T_i\Sigma_0 \lambda_i\\
\qquad\qquad\qquad+ b_iw^T\!(x)\mu_0
-p+a_i(x)-b_it+b_i^2 z\leq 0 ,\,\forall\, i=1,2,\cdots, m,\\
\left(
\begin{array}{cc}
4z &  w^T\!(x)\\
w(x)&Q \\
\end{array} \right)\succeq 0,
\end{array}
\right.
\end{eqnarray*}
where $z,s,t,p,v\in \mathbb{R}$, $x\in \mathbb{R}^n$, $\lambda_i \in \mathbb{R}^k$, and $H_1, H_2, V_1, V_2,Y\in \mathbb{R}^{k\times k}$ are decision variables for $i=1,2,\cdots,m$.

\end{example}

\begin{example}
Let
\begin{eqnarray*}
&&U_{1}=\left\{\zeta\in \mathbb{R}^n| \|\zeta\|_{2}\leq \rho_1\right\},\,U_{2}=\left\{\zeta\in \mathbb{R}^n| \|\zeta\|_{2}\leq \rho_2,\, \|\zeta\|_{1}\leq \rho_3 \right\},\\
&&Z_1=\left\{\Xi\in \mathbb{S}^n|\, \|\Xi\|_{F}\leq \theta_1  \right\},\,
Z_2=\left\{\Xi\in \mathbb{S}^n|\,\|\Xi\|_{F}\leq \theta_2,\|\Xi\|_{1\sigma}\leq \tau  \right\},\\
&&\varphi(\mu,\mu')=\beta_1 \|\mu-\mu' \|_{2},\,   \psi(\Sigma,\Sigma')=\beta_2\ln det(\Sigma-\Sigma'+E)^{-1},
\end{eqnarray*}
where $\beta_1> 0, \beta_2> 0, \rho_3\geq 0, \tau\geq 0, \rho_1\geq\rho_2\geq 0, \theta_1\geq \theta_2\geq 0$. Then it is easy to find that the support functions for $U_i$ and $Z_i$ have the following forms:
\begin{eqnarray*}
&&\delta^*\big(A\left(b_iw(x)\!-\!\lambda_i\!\right)|U_{1}\big)= \rho_1\|A\left(b_iw(x)\!-\!\lambda_i\!\right)\|_{2},\\
&&\delta^*\left(A\lambda_i|{U}_{2}\right)=\min_{u_{i1},u_{i2}}\left\{\rho_2\|u_{i1}\|_{2}+\rho_3 \|u_{i2}\|_{\infty} |\,u_{i1}+u_{i2}=A\lambda_i   \right\},\,\forall\,i=1,2,\cdots,m,\\
&&\delta^*\!\left(Q-Y |Z_{1}\right)=
\theta_1\left\|Q-Y\right\|_{F},  \\
&&\delta^*\!\left(Y|Z_{2}\right)=\min_{V_1,V_2}\left\{\theta_2\|V_1\|_{F}+\tau \delta_{\max}(V_2)|\,V_1+V_2=Y  \right\}.
\end{eqnarray*}
Moreover, we have
\begin{eqnarray*}
\varphi*(\lambda_i;-\lambda_i)&\!\!\!=\!\!\!&\left\{ \begin{array}{ll}\!\! 0, & \mbox{if} \; \|\lambda_i\|_{2}\leq \beta_1,\\ \!\!\infty, & \mbox{otherwise} \,. \end{array} \right.
\\
\psi^*\left(Y;-Y\right)
&\!\!\!=\!\!\!&\left\{ \begin{array}{ll}\!\! \beta_2\ln det(-Y)^{-1}\!+\!\beta_2(\ln\beta_2-1)n\!-\!tr(Y), & \mbox{if} \; Y \prec 0,\\ \!\!\infty, & \mbox{otherwise} \,. \end{array} \right.
\end{eqnarray*}
Let $H=-Y$. Then GDRC \eqref{th1gdr} can be rewritten as
\begin{eqnarray*}
\left\{
\begin{array}{l}
p+s\le 0,\,t^2-4zv\leq 0,\, z>0,\\
tr\left(Q^T\Sigma_0\right)
+\theta_1\left\|Q+H\right\|_{F}
+\theta_2\|V_1\|_{F}+\tau \delta_{\max}(V_2)\\
\qquad\qquad\qquad +\beta_2\ln det(H)^{-1}+\!\beta_2(\ln\beta_2-1)n\!+\!tr(H)+v-s
\leq 0,\\
 \rho_1\|A\left(b_iw(x)\!-\!\lambda_i\!\right)\|_{2}+\rho_2\|u_{i1}\|_{2}+\rho_3 \|u_{i2}\|_{\infty} \\
\qquad\qquad\qquad+ b_iw^T\!(x)\mu_0
-p+a_i(x)-b_it+b_i^2 z\leq 0 ,\,\forall\, i=1,2,\cdots, m,\\
V_1+V_2+H=0,\,H\succ 0 ,\,u_{i1}+u_{i2}=A\lambda_i,\,\|\lambda_i\|_{2}\leq \beta_1,\,\forall\, i=1,2,\cdots, m,\\
\left(
\begin{array}{cc}
4z &  w^T\!(x)\\
w(x)&Q \\
\end{array} \right)\succeq 0,
\end{array}
\right.
\end{eqnarray*}
where $z,s,t,p,v\in \mathbb{R}$, $x\in \mathbb{R}^n$, $u_{i1},u_{i2},\lambda_i \in\mathbb{R}^k$, $Q,V_1,V_2,H\in \mathbb{R}^{k\times k}$ are decision variables for $i=1,2,\cdots,m$.
\end{example}

Linear inequality has been widely used in optimization field. Next, we present a more tight reformulation for the globalized distributionally robust linear constraint and we find the reformulation do not contain the information of covariance.
\begin{cor}
Let $g(\xi,x)=a(x)+w^T\!(x)\xi$, where $a(x)$ and $w(x)$ are affine in $x$. Then GDRC \eqref{gdrc1+1} can be rewritten as follows:
\begin{eqnarray}\label{cor31+}
a(x)+w^T\!(x)\mu_0+\delta^*\big(A\left(w(x)\!-\!\lambda\!\right)|U_{1}\big)+\delta^*\left(A\lambda|{U}_{2}\right) +\varphi^*\left(\lambda;-\lambda\right)\leq 0.
\end{eqnarray}
where $\lambda\in \mathbb{R}^k,$ $x\in \mathbb{R}^n$ are decision variables.
Moreover, the feasible set of inequality \eqref{cor31+} is convex.
\end{cor}

\begin{proof}
When $m=1$, by the similar proof in Theorem \ref{gdrcth1}, the inner subordinate expectation problem $\sup_{{P}\in
\mathcal{P}_{11}}E_{P}\left[a(x)+w^T\!(x)\xi\right]$  is equivalent to
\begin{eqnarray*}
&&\inf_{p,t,z,s}\;\; p+s\\\nonumber
&&\;\mbox{s.t.}\quad\; \frac{w^T\!(x){\Sigma}\,w(x)+t^2}{4z}-s \leq 0,\;z\ge 0,\\\nonumber
&&\quad\quad\;\;\; w^T\!(x){\mu}-p+a(x)-t+ z\leq 0.
\end{eqnarray*}
It follows that
\begin{eqnarray*}
\inf_{p,t,z,s}\left\{ p+s\right\}\!\!\!\!\!&=&\!\!\!\!\!\inf_{t,z}\left\{a(x)+w^T\!(x){\mu}+ z+\frac{w^T\!(x){\Sigma}\,w(x)}{4z}+\frac{t^2}{4z}-t\right\}\\
&=& \inf_{z}\left\{a(x)+w^T\!(x){\mu}+ z+\frac{w^T\!(x){\Sigma}\,w(x)}{4z}- z\right\}\\
&=& a(x)+w^T\!(x){\mu}.
\end{eqnarray*}
Then, \eqref{gdrc2} with $m=1$ can be reformulated as
\begin{eqnarray*}
&&\sup_{
\mu\in \mathcal{U}_{\mu1},
\mu'\in \mathcal{U}_{\mu2}}\!\!\!
\left\{
a(x)+w^T\!(x)\mu
-\varphi(\mu,\mu')\}\right\}+\!\!\!\!
\sup_{
\Sigma\in \mathcal{U}_{\Sigma1},
\Sigma'\in \mathcal{U}_{\Sigma2}
}\!\!\!
\left\{
- \psi(\Sigma,\Sigma')\}\right\}
\leq 0.
\end{eqnarray*}
From $\psi(\Sigma,\Sigma')\ge 0$, it follows that $\sup_{
\Sigma\in \mathcal{U}_{\Sigma1},
\Sigma'\in \mathcal{U}_{\Sigma2}
}\!\!\!
\left\{
- \psi(\Sigma,\Sigma')\}\right\}=0$.
By the similar discussion in Theorem \ref{gdrcth1}, the globalized distributionally robust linear constraint \eqref{gdrc1}
can be rewritten as
\begin{eqnarray*}
a(x)+w^T\!(x)\mu_0+\delta^*\big(A\left(w(x)\!-\!\lambda\!\right)|U_{1}\big)+\delta^*\left(A\lambda|{U}_{2}\right) +\varphi^*\left(\lambda;-\lambda\right)\leq 0.
\end{eqnarray*}
This ends the proof. \end{proof}

Since the pointwise supremum of an arbitrary collection of convex function is convex, we notice that the feasible set of constraint \eqref{cor31+} is also convex when $a(x)$ and $w(x)$ are convex in $x$.

The ``normal range" set usually can be viewed as the set of the most possible realizations. A special case is that the inner ``normal range" set of realizations only has one element which is the empirical estimation, i.e., $\mathcal{U}_{\mu2}=\{\mu_0\}$ and $\mathcal{U}_{\Sigma2}=\{\Sigma_0\}$. Define 
$$\min_{\mu'\in \mathcal{U}_{\mu2}} \varphi(\mu,\mu')=\varphi(\mu,\mu_0)\triangleq \varphi(\mu),\quad
\min_{\Sigma'\in \mathcal{U}_{\Sigma2}} \psi(\Sigma,\Sigma')=\psi(\Sigma,\Sigma_0)\triangleq \psi(\Sigma).$$
Then it is easy to see that $\varphi(\mu)$ and $\psi(\Sigma)$  are two nonnegative convex functions with $\varphi(\mu_0)=0$ and $\psi(\Sigma_0)=0$.

In addition, we require that $\varphi(\mu)$ and $\psi(\Sigma)$  are lower semi-continuous.  Then we have the following corollary.
\begin{cor}\label{cor1}
When $\mathcal{U}_{\mu2}=\{\mu_0\}$ and $\mathcal{U}_{\Sigma2}=\{\Sigma_0\}$, GDRC \eqref{gdrc1+1} can be equivalently reformulated as the following system of inequalities
\begin{eqnarray}\label{fix}
\left\{
\begin{array}{l}
p+s\leq 0,\,t^2-4zv\leq 0,\,z>0,\\
\delta^*(Y|Z_{1})+\psi^*\left(Q-Y\right)+ tr(Y^T\Sigma_0)+v-s \leq 0,\\
\delta^*(A\lambda_i|U_{1})+\varphi^*\left(b_iw(x)-\lambda_i\right)+\lambda_i^T\mu_0
-p+a_i(x)-b_it+b_i^2 z\leq 0 ,\,\forall\, i=1,2,\cdots, m,\\
\left(
\begin{array}{cc}
4z &  w^T\!(x)\\
w(x)&Q \\
\end{array} \right)\succeq 0,
\end{array}
\right.
\end{eqnarray}
where $p,s,z,t,v\in \mathbb{R}, x\in \mathbb{R}^n$, $\lambda_i\in \mathbb{R}^k$, and $Y,Q\in \mathbb{R}^{k\times k}$
are decision variables for $i=1,2,\cdots,m$. Moreover, the feasible set of system \eqref{fix} is convex.
\end{cor}

\begin{proof}
By the similar discussion with the proof in Theorem \ref{gdrcth1}, GDRC \eqref{gdrc1+1}
in this case is equivalent to
\begin{eqnarray}
&&\inf_{p,t,z,s}\; p+s \nonumber\\\label{cd21}
&&\;\mbox{s.t.}\;\max_{\Sigma\in \mathcal{U}_{\Sigma1}} \left\{tr(Q^T\Sigma)-\psi(\Sigma)\right\}+v-s \leq 0,\\\label{cd22}
&&\quad\;\;\; \max_{\mu\in \mathcal{U}_{\mu1}} \left\{b_iw^T\!(x){\mu}-\varphi(\mu)\right\}-p+a_i(x)-b_it+b_i^2 z\leq 0 ,\,i=1,2,\cdots, m.\\ \nonumber
&&\quad\;\;\;\left(
\begin{array}{cc}
4z &  w^T\!(x)\\
w(x)&Q \\
\end{array} \right)\succeq 0,\,t^2-4zv\leq 0,\, z>0,
\end{eqnarray}

By Lemma \ref{Fenchel}, one has
\begin{eqnarray*}
&&\max_{\Sigma\in \mathcal{U}_{\Sigma1}} \left\{tr(Q^T\Sigma)-\psi(\Sigma)\right\}\\
&=&\!\!\max_{\Sigma} \left\{tr\left(Q^T{\Sigma}\right)-\psi(\Sigma)-\delta(\Sigma|\mathcal{U}_{\Sigma1}) \right\}\\
&=&\!\!\min_{Y} \left\{ \delta^*(Y|\mathcal{U}_{\Sigma1})-\left[tr\left(Q^TY\right)-\psi(Y)\right]_*
\right\},
\end{eqnarray*}
where
\begin{eqnarray*}
\left[tr\left(Q^TY\right)-\psi(Y)\right]_*\!\!\!\!&=&\!\!\!
\inf_{\Sigma}\left\{tr\left(Y^T\Sigma\right)-\left[tr\left(Q^T\Sigma\right)-\psi(\Sigma)\right]               \right\}\\
&=&\!\!\!  \inf_{\Sigma}\left\{ tr\!\!\left(\!\left(Y-Q\right)^T\!\Sigma\right)+\psi(\Sigma)\right\}\\
&=&\!\!\! -\sup_{\Sigma}\left\{tr\left( \left(Q-Y\right)^T\!\Sigma\right)-\psi(\Sigma)\right\}\\
&=&\!\!\!-\psi^*\left(Q-Y\right).
\end{eqnarray*}
By the discussion, constraint \eqref{cd21} holds if and only if $x,t,z,s$ together with variable $Y$ satisfy
\begin{eqnarray*}
\delta^*(Y|Z_{1})+\psi^*\left(Q-Y\right)+ tr(Y^T\Sigma_0)+v-s \leq 0.
\end{eqnarray*}

Similarly, the system of constraints \eqref{cd22} are equivalent to
\begin{eqnarray*}
\delta^*(A\lambda_i|U_{1})+\varphi^*\left(b_iw^T\!(x)-\lambda_i\right)+\lambda_i^T\mu_0
-p+a_i(x)-b_it+b_i^2 z\leq 0 ,\,\forall\, i=1,2,\cdots, m.
\end{eqnarray*}
This ends the proof.
\end{proof}

Actually, we notice that the conjugate functions $\psi^*\left(Y;-Y\!\right)$ with respect to both two variables has a few computationally tractable examples. When the inner ``normal range" set of realizations is described by the empirical estimation, the reformulation \eqref{fix} of the GDRC do not need to compute the convex conjugate function with respect to both two variables, which may have a computational advantage in comparison with these cases in which the inner ``normal range" sets have more than one element. Next we show a computationally tractable example for \eqref{fix}.

\begin{example}
Let
\begin{eqnarray*}
&&U_{1}=\left\{\zeta\in \mathbb{R}^k|\,h_{l}(\zeta)\leq 0,\,l=1,2,\cdots,L \right\},\\
&&Z_1=\left\{\Xi\in \mathbb{S}^k|\,tr(C_j\,\Xi)\leq c_j,\,j=1,2,\cdots,J\right\},\\
&&\varphi(\mu)=\sum_{\kappa=1}^{k}\mu_{\kappa}\ln(\frac{\mu_\kappa}{\mu_{0\kappa}}),\\
&&\psi(\Sigma)=
\left\{ \begin{array}{ll} \beta_2\left(tr(\Sigma^{-1}_0\Sigma)-k\right),\, & \mbox{if}\,\, \Sigma\succeq \Sigma_0,\\ 0,\, & otherwise. \end{array} \right.
\end{eqnarray*}
where $h_l(\cdot)$ is convex. In addition, if we assume that $\mu>0$, then
\begin{eqnarray*}
&&\delta^*(A\lambda_i|U_{1})=\min_{\theta_i}\left\{\sum^L_{l=1}\theta_{il} h^*_l\left(\frac{u_{il}}{\theta_{il}}\right)\,\Big|\theta_i\geq 0,\,\sum^L_{l=1}u_{il}=A\lambda_i \right\},\\
&&\delta^*\!\left(Y|Z_{1}\!\right)=
\min_{\eta_j}\left\{\sum_{j=1}^J c_j\eta_j  \Big|\sum_{j=1}^J \eta_jC_j=Y,\,\eta\geq 0 \right\},  \\
&&\varphi^*\left(b_iw(x)-\lambda_i\right)=\sum_{\kappa=1}^{k} \mu_{0i}e^{b_iw_\kappa(x)-\lambda_{i\kappa}-1},\\
&&\psi^*\left(Q-Y\right)=
\left\{ \begin{array}{ll}\!\! \beta_2k, & \mbox{if} \; \beta_2\Sigma^{-1}_0+Y-Q \succeq 0,\\ \!\!\infty, & \mbox{otherwise} \,. \end{array} \right.
\end{eqnarray*}
Thus, GDRC \eqref{fix} can be rewritten as
\begin{eqnarray*}
\left\{
\begin{array}{l}
p+s\leq 0,\,t^2-4zv\leq 0,\,z>0,\sum_{j=1}^J \eta_jC_j=Y,\,\eta\geq 0\\
\sum_{j=1}^J c_j\eta_j +\beta_2k+ tr(Y^T\Sigma_0)+v-s \leq 0,\\
\sum^L_{l=1}\theta_{il} h^*_l\left(\frac{u_{il}}{\theta_{il}}\right)+\sum_{\kappa=1}^{k} \mu_{0i}e^{b_iw_\kappa(x)-\lambda_{i\kappa}-1}+\lambda_i^T\mu_0
-p+a_i(x)-b_it+b_i^2 z\leq 0 ,\,\forall\, i=1,2,\cdots, m,\\
\theta_i\geq 0,\,\sum^L_{l=1}u_{il}=A\lambda_i,\,\forall\, i=1,2,\cdots, m,\\
\left(
\begin{array}{cc}
4z &  w^T\!(x)\\
w(x)&Q \\
\end{array} \right)\succeq 0,\,\beta_2\Sigma^{-1}_0+Y-Q \succeq 0,\,
\end{array}
\right.
\end{eqnarray*}
where $z,s,t,p\in \mathbb{R}$, $x\in \mathbb{R}^n$, $u_{i},\theta_i \in \mathbb{R}^L$, $\lambda_i\in \mathbb{R}^k$,
$\eta\in \mathbb{R}^J$, and $Y,Q\in \mathbb{R}^{k\times k}$ are decision variables for $i=1,2,\cdots,m$.
\end{example}

\section{GDRC with ${\rm(C2)}$ type distance function under second order moment information}
\noindent \setcounter{equation}{0}
In the former section, we consider the GDRC \eqref{GDRC001} under a separable distance function
about mean vector and covariance matrix and we find that the constraint can be rewritten as a deterministic system of inequalities. But for the distance function $H\big((\mu,\Sigma),(\mu',\Sigma')\big)$ which does not have the separable structure about mean and covariance, it is difficult to employ conjugate functions and support functions to obtain the deterministic reformulation of GDRC \eqref{GDRC001}. In this section, we discuss the GDRC \eqref{GDRC001} under a special distance function which is jointly convex in $(\mu,\Sigma)$. To be specific, supposing  $H\big((\mu,\Sigma),(\mu',\Sigma')\big)=\left(\mu-\mu_0\right)^T\Sigma^{-1}\left(\mu-\mu_0\right)$, we study the following GDRC
\begin{eqnarray} \label{jointgdrc}
E_{P\thicksim(\mu,\Sigma)}[\max_{i=1,2,\cdots,m}\{a_i(x)+b_iw^T\!(x)\xi\}]\leq \eta
\left(\mu-\mu_0\right)^T\Sigma^{-1}\left(\mu-\mu_0\right)
,\,\forall P\in \mathcal{P}_1,
\end{eqnarray}
where $\eta>0$, $a_i(x)$ and $w(x)$ are affine in $x$.
Then, we have the following result.

\begin{thm}\label{th3+}
GDRC \eqref{jointgdrc} can be reformulated as the following system of inequalities
\begin{eqnarray}\label{finaljointgdrc+}
\left\{
\begin{array}{l}
\delta^*\!\left(\!\left(Q+H_i\right)\big|Z_1\!\right)
\!+\!\delta^*\!\big(\!A\left(b_iw(x)\!+\!2h_i\right)|U_1\!\big)\!+\!tr( (Q+H_i)^T\Sigma_0)
+b_i w^T(x)\mu_0 \leq \varrho_i,\\
\varrho_i+v+a_i(x)-b_it+b_i^2 z\leq 0 ,z\ge 0,\,t^2-4zv\leq 0,\,h_{i0}\leq \eta,
\\
 \left(
\begin{array}{cc}
4z &  w^T\!(x)\\
w(x)&Q \\
\end{array} \right)\succeq 0,\,
 \left(
\begin{array}{cc}
H_i &  h_i\\
h_i^T & h_{i0}\\
\end{array} \right)\succeq 0,\,\,\forall\, i=1,2,\cdots, m,\\
\end{array}
\right.
\end{eqnarray}
where $t, z, \varrho_i, h_{i0}\in \mathbb{R}$, $h_i\in \mathbb{R}^k$, and $H_i \in \mathbb{R}^{k\times k}$ are decision variables for $i=1,2,\cdots, m$. Moreover, the feasible set of \eqref{finaljointgdrc+} is convex.
\end{thm}

\begin{proof}

Rewriting semi-infinite constraint \eqref{jointgdrc}, one has
\begin{eqnarray}\label{jointgdrc2}
\sup_{P\in \mathcal{P}_1}
\left\{
E_{P\thicksim(\mu,\Sigma)}\left[\max_{i=1,2,\cdots,m}\{a_i(x)+b_iw^T\!(x)\xi\}\right]
-\eta\left(\mu-\mu_0\right)^T\Sigma^{-1}\left(\mu-\mu_0\right)\right\}\leq 0.
\end{eqnarray}
Similarly, the left-hand side term of inequality \eqref{jointgdrc2} can be reformulated as
\begin{eqnarray}\label{+hh}
&&\sup_{\mu\in \mathcal{U}_{\mu1},
\Sigma\in \mathcal{U}_{\Sigma1}}\inf_{p,t,z,s}\;\; p+s-\eta\left(\mu-\mu_0\right)^T\Sigma^{-1}\left(\mu-\mu_0\right)\\\nonumber
&&\quad\quad\mbox{s.t.}\; \frac{w^T\!(x){\Sigma}\,w(x)+t^2}{4z}-s \leq 0,\;z\ge 0,\\\nonumber
&&\quad\quad\quad\quad\;b_iw^T\!(x){\mu}-p+a_i(x)-b_it+b_i^2 z\leq 0 ,\,\forall\,i=1,2,\cdots, m.
\end{eqnarray}
Reconstructing problem \eqref{+hh}, we have
\begin{eqnarray}\label{jointgdrc3}
&&\sup_{\mu\in \mathcal{U}_{\mu1},
\Sigma\in \mathcal{U}_{\Sigma1}}\inf_{p,t,z,\varrho_i}\;\;p\\\nonumber
&&\quad\quad\mbox{s.t.}\quad\;\frac{w^T\!(x){\Sigma}\,w(x)}{4z}+b_iw^T\!(x){\mu}-\eta\left(\mu-\mu_0\right)^T\Sigma^{-1}\left(\mu-\mu_0\right)\leq \varrho_i,\,\,\forall\,i=1,2,\cdots, m,\\\nonumber
&&\quad\quad\quad\quad\; \varrho_i+\frac{t^2}{4z}+a_i(x)-b_it+b_i^2 z\leq p ,\,z\ge 0,\,\forall\, i=1,2,\cdots, m.
\end{eqnarray}
By the same discussion as the proof of transforming \eqref{+max-min} into \eqref{+max-min2}, we obtain that problem \eqref{jointgdrc3} can be reformulated as
\begin{eqnarray}\label{jointsub}
&&\inf_{p,t,z,Q,\varrho_i}\;p\nonumber\\
&&\;\mbox{s.t.}\!\!\!\sup_{\mu\in \mathcal{U}_{\mu1},
\Sigma\in \mathcal{U}_{\Sigma1}}\!\left\{
tr(Q^T\Sigma)+b_iw^T\!(x){\mu}-\eta\left(\mu-\mu_0\right)^T\Sigma^{-1}\left(\mu-\mu_0\right)\right\}\leq \varrho_i,\,\forall\,i=1,2,\cdots, m,\\
&&\quad\quad\;\; \varrho_i+v+a_i(x)-b_it+b_i^2 z\leq p ,\,\forall\,i=1,2,\cdots, m,\nonumber\\\nonumber
&&\quad\quad\;\; \left(
\begin{array}{cc}
4z &  w^T\!(x)\\
w(x)&Q \\
\end{array} \right)\succeq 0,\,t^2-4zv\leq 0 ,\,z\ge 0.
\end{eqnarray}
Then, we notice that the subordinate maximization problem in constraint \eqref{jointsub} for each $i$ is equivalent to
\begin{eqnarray}\label{jointsub1}
&&\sup_{\mu\in \mathcal{U}_{\mu1},
\Sigma\in \mathcal{U}_{\Sigma1},r_i\ge0}\!
tr(Q^T\Sigma)+b_iw^T\!(x){\mu}-\eta r_i \\\nonumber
&&\quad\quad\mbox{s.t.} \;\; \left(\mu-\mu_0\right)^T\Sigma^{-1}\left(\mu-\mu_0\right)\leq r_i.
\end{eqnarray}
By Schur's complement, the constraint in \eqref{jointsub1} is equivalent to the following linear matrix inequality
 $$\left(
\begin{array}{cc}
\Sigma &  \left(\mu-\mu_0\right)\\
\left(\mu-\mu_0\right)^T & r_i\\
\end{array} \right)\succeq 0,$$
Then, the Lagrangian dual problem for \eqref{jointsub1} can be written as
\begin{eqnarray*}
&&\inf_{H_i,h_i,h_{i0}}\sup_{\mu\in \mathcal{U}_{\mu1},
\Sigma\in \mathcal{U}_{\Sigma1},r_i\ge 0}
tr(Q^T\Sigma)+b_iw^T\!(x){\mu}-\eta r_i\!+\!tr(H_i^T\Sigma)+2\left(\mu-\mu_0\right)^Th_i\!+\!h_{i0}r_i\\
&&\quad\mbox{s.t.} \; \left(
\begin{array}{cc}
H_i &  h_i\\
h_i^T & h_{i0}\\
\end{array} \right)\succeq 0,
\end{eqnarray*}
where $H_i\in \mathbb{R}^{k\times k},h_i\in \mathbb{R}^k,h_{i0}\in \mathbb{R}$ together form a matrix that is the dual variable associated with the constraint in \eqref{jointsub1}. By using the definition of support function, the above Lagrangian dual problem can
be simplified as
\begin{eqnarray*}
&&\inf_{H_i,h_i,h_{i0}}
\delta^*\!\left(\!\left(Q+H_i\right)\!|\mathcal{U}_{\Sigma1}\!\right)\!\!+\!
\delta^*\!\big(\!\left(b_iw\!(x)\!+\!2h_i\right)|\mathcal{U}_{\mu1}\!\big)
\!-\!2\mu_0^Th_i\\
&&\quad\mbox{s.t.} \; \left(
\begin{array}{cc}
H_i &  h_i\\
h_i^T & h_{i0}\\
\end{array} \right)\succeq 0,\,h_{i0}\leq \eta.
\end{eqnarray*}
Thus, problem \eqref{jointgdrc3} is equivalent to
\begin{eqnarray}\label{jointgdrc4}
&&\inf_{p,v,t,z,\varrho_i,H_i,h_i,h_{i0}}\;p\\\nonumber
&&\;\mbox{s.t.}\;
\delta^*\!\left(\!\left(Q+H_i\right)\!|\mathcal{U}_{\Sigma1}\!\right)\!\!+\!
\delta^*\!\big(\!\left(b_iw^T\!(x)\!+\!2h_i\right)|\mathcal{U}_{\mu1}\!\big)
\!-\!2\mu_0^Th_i\leq \varrho_i,\,\forall\,i=1,2,\cdots, m,\\\nonumber
&&\quad\quad\;\; \varrho_i+v+a_i(x)-b_it+b_i^2 z\leq p,\,z\ge 0,\,h_{i0}\leq \eta,\,t^2-4zv\leq 0 ,\,\,\forall\,i=1,2,\cdots, m,\\\nonumber
&&\quad\quad\;\;
\left(
\begin{array}{cc}
4z &  w^T\!(x)\\
w(x)&Q \\
\end{array} \right)\succeq 0,\,
\left(
\begin{array}{cc}
H_i &  h_i\\
h_i^T & h_{i0}\\
\end{array} \right)\succeq 0,
\,\forall\,i=1,2,\cdots, m.\nonumber
\end{eqnarray}
In addition,
\begin{eqnarray*}
&&\delta^*\!\left(\!\left(Q+H_i\right)\big|\mathcal{U}_{\Sigma1}\!\right)=
\!tr( (Q+H_i)^T\Sigma_0)+\delta^*\!\left(\!\left(Q+H_i\right)\big|Z_1\!\right),\,\forall\,i=1,2,\cdots, m,\\
&&\delta^*\!\big(\!\left(b_iw(x)\!+\!2h_i\right)|\mathcal{U}_{\mu1}\!\big)=
\left(b_iw(x)\!+\!2h_i\right)^T\mu_0+\delta^*\!\big(\!A\left(b_iw(x)\!+\!2h_i\right)|U_1\!\big),\,\forall\,i=1,2,\cdots, m.
\end{eqnarray*}
By substituting the above results into \eqref{jointgdrc4}, we find that GDRC \eqref{jointgdrc} can be written as reformulation \eqref{finaljointgdrc+}. This completes the proof.
\end{proof}

From Theorem \ref{th3+}, it is obvious that the equivalent system of the globalized distributionally robust optimization with ${\rm(C2)}$ type distance function under distribution set \eqref{P} only need to compute support functions $\delta^*\!\left(\cdot|Z_1\!\right)$ and $\delta^*(\cdot|U_1)$. Actually, many computationally tractable forms for $\delta^*\!(\cdot|U_1)$ can be found in \cite{BBHV2} and computationally tractable forms for $\delta^*\!\left(\cdot|Z_1\!\right)$ can be found in \cite{DHX}.
Next, we show a computationally tractable example for \eqref{finaljointgdrc+}.

\begin{example}
Let
$U_{1}=\left\{\zeta\in \mathbb{R}^n| C\zeta\leq c\right\},Z_1=\left\{\Xi\in \mathbb{S}^n|\,\|\Xi \|_{\sigma p} \!\leq\! \tau_2 \right\}$,
where $\theta \geq 0$, $c\in \mathbb{R}^k$, and $C\in \mathbb{R}^{L\times k}$.
By $tr(X^TY)\leq \|X\|_{\sigma p}\|Y\|_{\sigma q}$, one has
\begin{eqnarray*}
&&\delta^*\!\left(\!\left(Q+H_i\right)\big|Z_1\!\right)=\|Q+H_i\|_{\sigma q},\,\forall\,i=1,2,\cdots, m,
\\
&&\delta^*\!\big(\!A\left(b_iw(x)\!+\!2q_i\right)|U_1\!\big)=\min_{u_i\geq 0}
\left\{c^T u_i|\,C u_i=A\left(b_iw(x)\!+\!2q_i\right)\right\},\,\forall\,i=1,2,\cdots, m,
\end{eqnarray*}
where $\frac{1}{p}+\frac{1}{q}=1$.
By Theorem \ref{th3+}, GDRC \eqref{jointgdrc} can be rewritten as
\begin{eqnarray*}
\left\{
\begin{array}{l}
\|Q+H_i\|_{\sigma q}
+c^T u_i+tr( (Q+H_i)^T\Sigma_0)
+b_i^T w(x)\mu_0 \leq \varrho_i,\,\forall\,i=1,2,\cdots, m,\\
\varrho_i+v+a_i(x)-b_it+b_i^2 z\leq 0 ,z> 0,\,t^2-4zv\leq 0,\,h_{i0}\leq \eta,\,\forall\,i=1,2,\cdots, m,
\\
C u_i=A\left(b_iw(x)\!+\!2q_i\right),\,u_i> 0,\,\forall\,i=1,2,\cdots, m,\\
 \left(
\begin{array}{cc}
4z &  w^T\!(x)\\
w(x)&Q \\
\end{array} \right)\succeq 0,\,
 \left(
\begin{array}{cc}
H_i &  h_i\\
h_i^T & h_{i0}\\
\end{array} \right)\succeq 0,\,\forall\,i=1,2,\cdots, m,\\
\end{array}
\right.
\end{eqnarray*}
where $z,t,\varrho_i, h_{i0}\in \mathbb{R}$, $x\in \mathbb{R}^n$, $u_{i}, h_i\in \mathbb{R}^k$, and $Q,H_i\in \mathbb{R}^{k\times k}$ are decision variables for $i=1,2,\cdots, m$.
\end{example}

\section{GDRC with ${\rm(C3)}$ type distance function under first order moment and support information }
\noindent \setcounter{equation}{0}
In the former discussion, we do not impose restriction on support information and let $\xi\in \mathbb{R}^k$. In
many practical problems the support of the distribution $P$ about $\xi$ is known to be a strict subset of $\mathbb{R}^k$.
We notice that the support information of distribution $P$ will result in unnecessarily conservation about distributionally
robust constraint. Here we consider the following distribution set under support information and first moment information
\begin{eqnarray}\label{p2}
\mathcal{P}_2
=\left\{
P\in
\mathfrak{P}: P({\xi}\in \mathcal{U}_{\xi}
)=1,\,
\mu\in \mathcal{U}_{\mu1}
\right\},
\end{eqnarray}
where $\mathcal{U}_{\xi}$ and $\mathcal{U}_{\mu1}$ are convex and compact sets. Since distribution set \eqref{p2} do not contain covariance information, we consider the following GDRC
\begin{eqnarray}\label{gdrc1support}
E_{P}[g(\xi,x)]\leq
\min_{\mu'\in \mathcal{U}_{\mu2}} \varphi(\mu,\mu')
,\,\forall P\in \mathcal{P}_2,
\end{eqnarray}
where $g(\cdot,x)$ is proper concave and upper-semicontinuous for all ${x\in \mathbb{R}^n}$, and $\varphi(\mu,\mu')$ is a nonnegative jointly convex function with $\varphi(\mu,\mu)=0$.

\begin{thm}\label{th1.3+}
GDRC \eqref{gdrc1support} under information set \eqref{p2} is equivalent to the following inequality
\begin{eqnarray}\label{sf+}
 \delta^*(w|\mathcal{U}_{\xi})-g_*(w+s_1,x)+
\delta^*(A(s_1-\theta)|U_{1})+\varphi^*(\theta;-\theta)+\delta^*(A\theta|{U}_{2})+\mu_0^Ts_1
\leq 0.
\end{eqnarray}
where $w,s_1,\theta\in \mathbb{R}^k$, $x\in \mathbb{R}^n$ are decision variables. If $g(\xi,x)$ is convex in $x$ for all ${\xi\in \mathbb{R}^n}$, then the feasible set of \eqref{th1.3+} is convex.
\end{thm}

\begin{proof}
Firstly, we define the following distribution information set
\begin{eqnarray*}
\mathcal{P}_{21}\left(\mathcal{U}_{\xi},{\mu}\right)=\left\{\begin{array}{ccc}
P: P({\xi}\in
\mathcal{U}_{\xi})=1,\, \mathbb{E}[{\xi}]={\mu}
\end{array}\right\}.
\end{eqnarray*}
In set $\mathcal{P}_{21}$, we suppose the first moment is given.
Then, \eqref{gdrc1support} can be reformulated as follows:
\begin{eqnarray}\label{supsupdrsupport3}
\sup_{
\mu\in \mathcal{U}_{\mu1},\mu'\in \mathcal{U}_{\mu2}
}
\left\{
\sup_{{P}\in
\mathcal{P}_{21}}
E_{P}\left[g(\xi,x)\right]-\varphi(\mu,\mu')
\right\}\leq 0.
\end{eqnarray}
The inner subordinate problem
$\sup_{{P}\in
\mathcal{P}_{21}}
E_{P}\left[g(\xi,x)\right]$ can be unfolded as follows:
\begin{eqnarray*}
&&\sup_{\zeta\in \mathcal{M}_+}\;
\int_{\mathbb{R}}g(\xi,x) \zeta(d\xi)\\
&&\;\;\mbox{s.t.}\;\; \int_{\mathcal{U}_{\xi}}\zeta(d\xi)=1,\; \int_{\mathcal{U}_{\xi}} \xi\,\zeta(d\xi)=\mu
\end{eqnarray*}
and so the dual is
\begin{eqnarray}\label{drorg1.1}
&&\inf_{s_0,s_1}\;
s_0+s_1^T\mu\\\nonumber
&&\,\mbox{s.t.}\;\, s_0\geq g(\xi,x)-s_1^T\xi ,\; \forall \xi\in \mathcal{U}_{\xi},
\end{eqnarray}
where $s_0\in \mathbb{R},s_1\in \mathbb{R}^n$ are the dual variables of corresponding constraints.
The robust counterpart of the constraint in problem \eqref{drorg1.1} is
\begin{eqnarray*}
s_0\geq \max_{\xi\in \mathcal{U}_{\xi}}\left\{ g(\xi,x)-s_1^T\xi \right\}.
\end{eqnarray*}
By Lemma \ref{Fenchel}, we have
\begin{eqnarray*}
&&\max_{\xi\in \mathcal{U}_{\xi}}\left\{ g(\xi,x)-s_1^T\xi \right\}\\
&=& \max_{\xi}\left\{g(\xi,x)-s_1^T\xi-\delta(\xi|\mathcal{U}_{\xi}) \right\}\\
&=&\min_{w} \delta^*(w|\mathcal{U}_{\xi})-[g(w,x)-s_1^T w]_*\\
&=&\min_{w} \delta^*(w|\mathcal{U}_{\xi})-g_*(w+s_1,x).
\end{eqnarray*}

\noindent Then~\eqref{drorg1.1} can be simplified as
\begin{eqnarray*}
&&\min_{w,s_1}\; \delta^*(w|\mathcal{U}_{\xi})-g_*(w+s_1,x)+s_1^T\mu.
\end{eqnarray*}
Thus, problem~\eqref{supsupdrsupport3} becomes
\begin{eqnarray}\label{saddlept}
\sup_{
\mu\in \mathcal{U}_{\mu1},\mu'\in \mathcal{U}_{\mu2}
}
\left\{
\min_{w,s_1}\; \delta^*(w|\mathcal{U}_{\xi})-g_*(w+s_1,x)+s_1^T\mu-\varphi(\mu,\mu')
\right\}\leq 0.
\end{eqnarray}
By the similar discussion which transforms \eqref{+max-min} into \eqref{+max-min2+} in Theorem \ref{gdrcth1},  the left hand term of \eqref{saddlept} is equivalent to
\begin{eqnarray*}
\min_{w,s_1}\sup_{
\mu\in \mathcal{U}_{\mu1},\mu'\in \mathcal{U}_{\mu2}
}
\left\{
\; \delta^*(w|\mathcal{U}_{\xi})-g_*(w+s_1,x)+s_1^T\mu-\varphi(\mu,\mu')
\right\}.
\end{eqnarray*}
Note that $\sup_{
\mu\in \mathcal{U}_{\mu1},\mu'\in \mathcal{U}_{\mu2}
}\left\{ s_1^T\mu-\varphi(\mu,\mu')
\right\}$ can be rewritten as
 \begin{eqnarray*}
&&\min_{\theta}
\left\{
\delta^*(A(s_1-\theta)|U_{1})+\varphi^*(\theta;-\theta)+\delta^*(A\theta|{U}_{2})+\mu_0^Ts_1
\right\}.
\end{eqnarray*}
Thus, inequality \eqref{gdrc1support} can be reformulated as the following problem
\begin{eqnarray*}
&&\min_{w,s_1,\theta}
\left\{  \delta^*(w|\mathcal{U}_{\xi})-g_*(w+s_1,x)+
\delta^*(A(s_1-\theta)|U_{1})+\varphi^*(\theta;-\theta)+\delta^*(A\theta|{U}_{2})+\mu_0^Ts_1
\right\}\leq 0.
\end{eqnarray*}
This shows that \eqref{th1.3+} holds. Next, we will prove that $g_*(v,x)$ is concave in $(v,x)$ when $g(\xi,x)$ is convex in $x$ for all ${\xi\in \mathbb{R}^k}$. In fact, for any $(v_1, x_1)$, $(v_2, x_2)$ and $t\in [0,1]$, we have
\begin{eqnarray*}
&&tg_*(v_1,x_1)+(1-t)g_*(v_2,x_2)\\
&=&t\inf_{\xi \in \mathbb{R}^k}\left\{v_1^T\xi-g_(\xi,x_1)\right\}+(1-t)\inf_{\xi \in \mathbb{R}^k}\left\{v_2^T\xi-g_(\xi,x_2) \right\}\\
&\leq& \inf_{\xi \in \mathbb{R}^k} \left\{ (tv_1+(1-t)v_2)^T\xi-(tg_(\xi,x_1)+(1-t)g_(\xi,x_2))  \right\}\\
&\leq& \inf_{\xi \in \mathbb{R}^k} \left\{ (tv_1+(1-t)v_2)^T\xi-g(\xi,tx_1+(1-t)x_2) \right\}\\
&=& g_*(tv_1+(1-t)v_2,tx_1+(1-t)x_2).
\end{eqnarray*}
Combing the above discussion with the fact that the support functions and conjugate function are convex, we conclude that the feasible set is convex when $g(\xi,x)$ is convex in $x$ for all ${\xi\in \mathbb{R}^k}$.
\end{proof}

From Theorem \ref{th1.3+}, we notice that the computations involving $\mathcal{U}_{\xi}, {U}_{1}, U_{2},\varphi$ and $g$
are all separable. More details about the computation of $g_*(\cdot,x)$, please see \cite{BT}. Actually, \eqref{sf+} can be
unfolded in a computationally tractable way for variety choice of $\mathcal{U}_{\xi}, {U}_{1}, U_{2},\varphi$ and $g$.
Next we will show an example.

\begin{example}
Let
\begin{eqnarray*}
&&U_{i}=\left\{\zeta\in \mathbb{R}^n|\,\|\zeta\|_{2}\leq \tau_i \right\},\;i=1,2,\\
&&U_{\xi}=\left\{\xi\in \mathbb{R}^m|\,\|\xi\|_{1}\leq \rho_1,\,\|\xi\|_{2}\leq \rho_2,\,\|\xi\|_{\infty}\leq \rho_3, \right\},\\
&&\varphi(\mu,\mu')=\alpha\|\mu-\mu'\|_{p},\,   g(\xi,x)=\xi^Tx+\vartheta,
\end{eqnarray*}
where $\tau_1\geq \tau_2 \geq 0, \rho_{1}\geq 0, \rho_{2}\geq 0, \rho_{3}\geq 0, \alpha\geq 0$.
By simply computation, we have
\begin{eqnarray*}
&&\delta^*\big(A\left(s_1-\theta  \right)|U_{1}\big)=\rho_1\|A\left(s_1-\theta\right)\|_{2},\\
&&\delta^*\left(A\theta|{U}_{2}\right)=\rho_2\|A\theta\|_{2},\\
&&\delta^*\left(w|U_{\xi}\right)=
\min_{y_1,y_2,y_3}\left\{\rho_1\|y_1\|_{\infty}+\rho_2\|y_2\|_{2}+\rho_3\|y_3\|_{1}\,
 |\, y_1+y_2+y_3=w  \right\},\\
&&\varphi^*\left(\theta;-\theta\right)=
\left\{ \begin{array}{ll} 0, & \mbox{if}\, \|\lambda\|_{q}\leq \alpha,\\ +\infty, & \mbox{otherwise} \,. \end{array} \right.
\\
&&g_*\left(w+s_1,x\right)=
\left\{ \begin{array}{ll} 0, & \, w+s_1=x,\\ +\infty, &\, w+s_1\neq x. \end{array} \right.
\end{eqnarray*}
Here $\|\cdot\|_{q}$ is the dual norm of $\|\cdot\|_{p}$ with $\frac{1}{p}+\frac{1}{q}=1$. Thus, GDRC \eqref{gdrc1support} becomes
\begin{eqnarray*}
&&\rho_1\|y_1\|_{\infty}+\rho_2\|y_2\|_{2}+\rho_3\|y_3\|_{1}
+
\rho_1\|A\left(s_1-\theta\right)\|_{2}+
+\rho_2\|A\theta\|_{2}+\mu_0^Ts_1
\leq 0,\\
&&y_1+y_2+y_3=w,\, w+s_1=x,\,\|\lambda\|_{q}\leq \alpha,\,
\end{eqnarray*}
where $w, s_1, \theta, y_1, y_2, y_3, \lambda\in \mathbb{R}^k$, $x\in \mathbb{R}^n$
are decision variables.
\end{example}

By the similar proof of Corollary \ref{cor1}, we can derive the following result.

\begin{cor}
Suppose $\mathcal{U}_{\mu2}=\{\mu_0\}$, then $\min_{\mu'\in \mathcal{U}_{\mu2}} \varphi(\mu,\mu')=\varphi(\mu,\mu_0)\triangleq \varphi(\mu)$, where  $\varphi(\mu)$ is a nonnegative proper convex and lower-semicontinuous function with $\varphi(\mu_0)=0$.
Then, GDRC \eqref{gdrc1support} can be equivalently reformulated as
\begin{eqnarray*}
 \delta^*(w|\mathcal{U}_{\xi})-g_*(w+s_1,x)+
\delta^*(A\theta|U_{1})+\varphi^*(s_1-\theta)+\mu_0^T\theta
\leq 0,
\end{eqnarray*}
where $w,s_1,\theta\in \mathbb{R}^k$, $x\in \mathbb{R}^n$ are decision variables.
\end{cor}

\section{A numerical example}
In this section, we consider a distributionally robust portfolio selection problem with the Conditional Value-at-Risk as the risk measure to illustrate the efficiency of GDRC. Conditional Value-at-Risk (CVaR), defined as the mean of the tail distribution exceeding VaR, has become more and more popular  in financial management recently. As a coherent risk measure, CVaR can be reformulated as a convex program. We recall the definition of CVaR at level $\epsilon$ with respect to $P$ for any fixed ${x}$ in Rockafellar and Uryasev\cite{RU}:
\begin{eqnarray}\label{cvar}
\mbox{$P$-CVaR}_{\epsilon}\left(L({x},{\xi})\right)=\inf_{\beta\in
\mathbb{R}} \left\{\beta+\frac{1}{\epsilon}\mathbb{E}_{P}\left[(L({x},{\xi})-\beta)^+\right]\right\},
\end{eqnarray}
where $L({x},{\xi})=-x^T\xi$ denotes the loss function associated with the allocation vector $x\in \mathbb{R}^3$ and the random stocks' returns vector $\xi\in \mathbb{R}^3$. To be specific, $x^T\xi$ is the portfolio return. Three constituent stocks in the financial market have been chosen by the investor. The empirical mean vector and covariance matrix of the three stocks' returns are estimated as follows:
\begin{eqnarray*}
\mu_0=\begin{pmatrix}
0.0409\\
0.0854\\
0.0702
\end{pmatrix},
\quad
\Sigma_0=
  \begin{pmatrix}
0.0075& 0.0065 & 0.0080\\
0.0065 &0.0149 & 0.0089\\
    0.0073 & 0.0089 & 0.0121
  \end{pmatrix}.
\end{eqnarray*}
Actually, it is hard to acquire the precise knowledge of the distribution $P$ in practice. Without loss of generality, we assume $P\in \mathcal{P}_1$. A distributionally robust portfolio selection problem using worst case CVaR as a risk measure can be represented as
\begin{eqnarray}\label{s5cvar1}
&&\min_{x} \;\;\sup_{P\in \mathcal{P}_1}\mbox{$P$-CVaR}_{\epsilon}\left(-x^T\xi\right)\\\nonumber
&&\;\mbox{s.t.}\quad x^Te=1, x\ge 0.\nonumber
\end{eqnarray}
We do not require a lower limit on the expected return since the CVaR reformulation can be viewed as the risk-adjusted expected return form \cite{HZFF}. By introducing a new variable $v$,  similar discussion in \cite{ZF}, \eqref{s5cvar1} can be reformulated as
\begin{eqnarray}\label{s5cvar2}
&&\min_{x,\beta,v} \;\;v\\\nonumber
&&\;\mbox{s.t.}\quad \beta+\frac{1}{\epsilon}\mathbb{E}_{P}\left[(-x^T\xi-\beta)^+\right]\le v,\,\forall {P\in \mathcal{P}_1},\\\nonumber
&&\quad\quad\;\; x^Te=1, x\ge 0.
\end{eqnarray}
The worst case analysis would be performed and the optimal strategy would be conservative. To enhance the flexibility, the globalized distributionally robust portfolio selection problem can be represented as
\begin{eqnarray}\label{gs5cvar2}
&&\min_{x,\beta,v} \;\;v\\\nonumber
&&\;\mbox{s.t.}\quad \beta+\frac{1}{\epsilon}\mathbb{E}_{P}\left[(-x^T\xi-\beta)^+\right]\le v+\inf_{(\mu',\Sigma')\in \mathcal{U}_{(\mu'\!,\Sigma')}}\!\!\!\!\!\!\! H\left((\mu,\Sigma),(\mu',\Sigma')\right),\,\forall P\in \mathcal{P}_1,\\\nonumber
&&\quad\quad\;\; x^Te=1, x\ge 0.
\end{eqnarray}
In the following, we compare the performance of portfolio selection problem \eqref{s5cvar2} and \eqref{gs5cvar2} by using the approaches given in Section 3. All optimization problems are solved using
the YALMIP modeling language and MOSEK 9 (see \cite{L}).

\begin{figure}[h]\label{4cvarfig1}
\centering
\includegraphics[width=7.8cm]{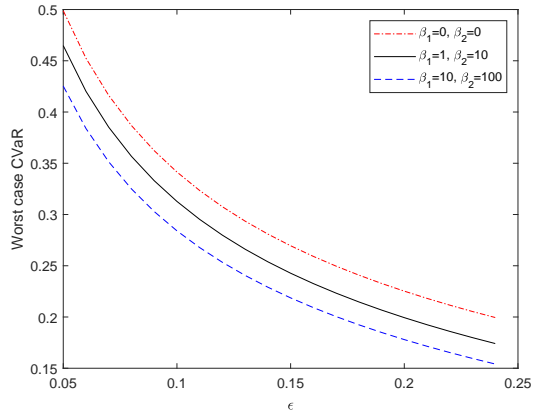}
\caption{The relation between the value of VaR and $\epsilon$ }
\end{figure}

In the subsequent tests, we set
\begin{eqnarray*}
U_{i}=\left\{\zeta\in \mathbb{R}^n|\,\|\zeta\|_{2}\leq \rho_i \right\}, \quad
Z_{i}=\left\{\Xi\in \mathbb{S}^n|\,0\preceq\Xi\preceq \tau_{i}\Sigma_0 \right\},\quad i=1,2.
\end{eqnarray*}
Assume that
\begin{eqnarray*}
\varphi(\mu,\mu')=\frac{\beta_1}{2}(\mu-\mu')^T\Sigma^{-1}_0(\mu-\mu'),\,
\psi(\Sigma,\Sigma')=\frac{\beta_2}{2}tr\left(\left(\Sigma-\Sigma'\right)^T\left(\Sigma-\Sigma'\right) \right)
\end{eqnarray*}

\begin{table}[!hbp]
\caption{$x$ under different value of $\beta_1$.}
\vspace{3mm}
\centering
\begin{tabular}{|c| c |c| c|}
\hline
   $\beta_2=50$  & worst case CVaR   &  $x$    \\
\hline
$\beta_1=0.1$  & 0.2072 &(0.5947, 0.2816, 0.1237)  \\
$\beta_1=1$ & 0.1888  &(0.5548, 0.2919, 0.1533) \\
$\beta_1=5$  &  0.1856 & (0.5397, 0.2958, 0.1645) \\
$\beta_1=10$ & 0.1852  &(0.5377, 0.2963, 0.1660) \\
$\beta_1=50$ & 0.1848  &(0.5361, 0.2967,  0.1671) \\
 \hline
\end{tabular}
\label{table2}
\end{table}
\begin{table}[!h]
\caption{$x$ under different value of $\beta_2$.}
\vspace{3mm}
\centering
\begin{tabular}{|c| c |c| c|}
\hline
   $\beta_1=5$  & worst case CVaR   &  $x$    \\
\hline
$\beta_1=1$  & 0.1995 &(0.7500, 0.2500, 0.0000)  \\
$\beta_1=10$ & 0.1963  &(0.7309, 0.2629,0.0061) \\
$\beta_1=30$  &  0.1903 & (0.6318, 0.2812,0.0870) \\
$\beta_1=50$ & 0.1856  &(0.5397,0.2958,0.1645) \\
$\beta_1=70$ & 0.1820  &(0.4703,0.3049,0.2248) \\
 \hline
\end{tabular}
\label{table2}
\end{table}

Let $A=\Sigma_0^{\frac{1}{2}}, \rho_1=0.5,\rho_2=0.2,\tau_1=0.8$ and $\tau_2=0.3$. Computational experiments on three different random cases are conducted to evaluate the performance of the CVaR. Figure 6.1 shows the performance of the value of CVaR under different values of parameter $\epsilon$. We notice that the red line under case $\beta_1=0, \beta_2=0$ represents the worst case CVaR under distribution set $\mathcal{P}_1$ and the black and blue lines under $\beta_1=1, \beta_2=10$ and $\beta_1=10, \beta_2=50$ represent the globalized worst case CVaR, respectively. When $\epsilon$ becomes large which means that the investor enhances the minimal safety level, the mean of potential portfolio loss decreases under three cases. Compared to distributionally robust portfolio model, the value of CVaR of the globalized distributionally robust portfolio model is smaller and the investor may hold a more optimistic view on the financial market. If the investor holds that the financial market becomes fine and exciting, they can reduce the conservatism of the model by adjusting the parameters of distance function.

Table 6.1 and Table 6.2 show the performance of the portfolio strategies under different values of parameter $\beta_1$ and $\beta_2$, respectively. When $\beta_1$ (resp., $\beta_2$) becomes large which means the investors want to reduce the conservatism of the globalized distributionally robust portfolio model, the investor will invest less money to the lower risk asset and put more money in the second and third stocks which represent the higher risk assets. These facts indicate that the strategy of portfolio obtained by the globalized distributionally robust portfolio model is not conservative enough and so we can carry out more flexible and nonconservative strategies by adjusting the parameters of the distance function. Compared to the original distributionally robust portfolio model, the globalized distributionally robust portfolio model really improves the conservativeness which can not be neglected in the study of the distributionally robust optimization.

\section{Conclusion}
In this paper, we develop the GDRC in terms of the moment-based framework to reduce the conservatism of the distributionally robust counterpart. We obtain the deterministic equivalent reformulations for GDRCs under convex second order moment information with a separable convex distance function and a special function which is jointly in mean and covariance, respectively.  Moreover, the reformulation for the GDRC under assumptions that the first order moment set and the support set are convex and compact with the constraint function which is concave in uncertain parameter has also been presented. In each case, we show the computationally tractable examples for the corresponding reformulation. The numerical test shows that the GDRC is more flexible with respect to both the uncertainty distribution sets and the distance function.

In the future, it would be important and interesting to study the GDRCs under more general moment information set, i.e., second order moment and support information, first order moment and subsets of support information and so on. Another open direction is to study the GDRCs with more general constraint function, i.e., the GDRC under second order moment information with the constraint function which is convex in uncertainty parameter. Also, we would like to explore more computationally tractable examples for conjugate function on $\mathbb{R}^{k\times k}$ to enhance the application of the GDRC. Recently, variety types of distributionally robust optimization problems under information set described by Wasserstein metric and Kullback-Leibler divergence have been discussed by many researchers, it would be interesting to study the GDRC under information set described by Wasserstein metric or Kullback-Leibler divergence.

\end{document}